\documentclass{amsart}
\usepackage{todonotes}
\usepackage{mathpazo}
\usepackage{enumerate}
\usepackage{subcaption}
\usepackage{tikz,pgf}
	\pgfdeclarelayer{background}
	\pgfdeclarelayer{middle}
	\pgfsetlayers{background,middle,main}
\usepackage[colorlinks=true,
	urlcolor=blue!20!black,
	citecolor=green!50!black,pagebackref=true]{hyperref}
\usepackage{amssymb}
\usepackage[nobysame]{amsrefs}
\usepackage{ulem}
\makeatletter
\let\orgdescriptionlabel\descriptionlabel
\renewcommand*{\descriptionlabel}[1]{%
  \let\orglabel\label
  \let\label\@gobble
  \phantomsection
  \edef\@currentlabel{#1}%
  \let\label\orglabel
  \orgdescriptionlabel{(#1)}%
}
\newtheorem{theorem}{Theorem}[section]
\newtheorem{conjecture}[theorem]{Conjecture}
\newtheorem{proposition}[theorem]{Proposition}
\newtheorem{lemma}[theorem]{Lemma}
\newtheorem{corollary}[theorem]{Corollary}

\theoremstyle{remark}
\newtheorem{example}[theorem]{Example}
\newtheorem{remark}[theorem]{Remark}

\newcommand{\defn}[1]{{\color{green!50!black}\textit{#1}}}
\newcommand{\defs}{\stackrel{\mathrm{def}}{=}}
\newcommand{\ie}{\text{i.e.}\;}
\renewcommand{\th}{^{\mathrm{th}}}

\newcommand{\Symmetric}{\mathfrak{S}}
\newcommand{\oneline}[1]{#1}
\newcommand{\Inv}{\mathsf{Inv}}
\newcommand{\Des}{\mathsf{Des}}
\newcommand{\Weak}{\textbf{\textsf{Weak}}}
\newcommand{\Tamari}{\textbf{\textsf{Tam}}}
\newcommand{\id}{\textsf{e}}
\newcommand{\wo}{w_{o}}
\newcommand{\woa}{w_{o;\alpha}}
\newcommand{\NC}{\mathsf{Nonc}}
\newcommand{\NN}{\mathsf{Nonn}}
\newcommand{\proj}{\pi^{\downarrow}}
\newcommand{\JI}{\mathsf{JoinIrr}}
\newcommand{\MI}{\mathsf{MeetIrr}}
\newcommand{\perspective}{\doublebarwedge}
\newcommand{\Poset}{\mathbf{P}}
\newcommand{\Lattice}{\mathbf{L}}
\newcommand{\dual}{\textsf{d}}
\newcommand{\Covers}{\mathcal{E}}
\newcommand{\least}{\hat{0}}
\newcommand{\grtst}{\hat{1}}

\newcommand{\Con}{\mathsf{Con}}
\newcommand{\cg}{\mathsf{cg}}
\newcommand{\Galois}{\mathsf{Galois}}
\newcommand{\CLO}{\textbf{\textsf{CLO}}}
\newcommand{\Pf}{\mathbf{P}}
\newcommand{\refines}{\leq_{\mathsf{ref}}}
\newcommand{\BijNCAlign}{\Phi}
\BibSpec{collection.article}{%
    +{}  {\PrintAuthors}                {author}
    +{,} { \textit}                     {title}
    +{.} { }                            {part}
    +{:} { \textit}                     {subtitle}
    +{,} { \PrintContributions}         {contribution}
    +{,} { \PrintConference}            {conference}
    +{}  {\PrintBook}                   {book}
    +{,} { }                            {booktitle}
    +{}  { \PrintEditorsB}              {editor}
    +{,}  { }                           {publisher}
    +{,}  { }                           {address}
    +{,} { \PrintDateB}                 {date}
    +{,} { pp.~}                        {pages}
    +{,} { }                            {status}
    +{,} { \PrintDOI}                   {doi}
    +{,} { available at \eprint}        {eprint}
    +{}  { \parenthesize}               {language}
    +{}  { \PrintTranslation}           {translation}
    +{;} { \PrintReprint}               {reprint}
    +{.} { }                            {note}
    +{.} {}                             {transition}
    +{}  {\SentenceSpace \PrintReviews} {review}
}
\title[Noncrossing Arc Diagrams, Tamari Lattices, and Parabolic Quotients]{Noncrossing Arc Diagrams, Tamari Lattices, and Parabolic Quotients of the Symmetric Group}
\author{Henri M{\"u}hle}
\address{Technische Universit{\"a}t Dresden, Institut f{\"u}r Algebra, Zellescher Weg 12--14, 01069 Dresden, Germany.}
\email{henri.muehle@tu-dresden.de}

\begin{document}

\allowdisplaybreaks

\normalem

\begin{abstract}
	Ordering permutations by containment of inversion sets yields a fascinating partial order on the symmetric group: the weak order.  This partial order is, among other things, a semidistributive lattice.  As a consequence, every permutation has a canonical representation as a join of other permutations.  Combinatorially, these canonical join representations can be modeled in terms of arc diagrams.  Moreover, these arc diagrams also serve as a model to understand quotient lattices of the weak order.  
	
	A particularly well-behaved quotient lattice of the weak order is the well-known Tamari lattice, which appears in many seemingly unrelated areas of mathematics.  The arc diagrams representing the members of the Tamari lattices are better known as noncrossing partitions.  Recently, the Tamari lattices were generalized to parabolic quotients of the symmetric group.  In this article, we undertake a structural investigation of these parabolic Tamari lattices, and explain how modified arc diagrams aid the understanding of these lattices.
\end{abstract}

\maketitle

\section{Introduction}
	\label{sec:introduction}
Given a permutation $w$ of $[n]\defs\{1,2,\ldots,n\}$, an \emph{inversion} of $w$ is a pair of indices for which the corresponding values of $w$ are out of order.  In other words, the number of inversions of $w$ is a measure of \emph{disorder} introduced by $w$.  A permutation is characterized by its inversion set, \ie the set of pairs encoding the locations of the inversions.

Containment of inversion sets introduces a partial order on the set $\Symmetric_{n}$ of \emph{all} permutations of $[n]$; the \emph{weak order}.  This partial order has many remarkable properties.  For instance, it is a lattice~\cites{yanagimoto69partial,guilbaud71analyse}.  The \emph{diagram} of the weak order is the graph on the vertex set $\Symmetric_{n}$ in which two permutations are related by an edge if they differ by swapping a \emph{descent}, \ie an inversion whose corresponding values are adjacent integers.  By construction, this diagram is isomorphic to the $1$-skeleton of the permutohedron~\cite{gaiha77adjacent}. 

Perhaps an even more remarkable property of the weak order on $\Symmetric_{n}$ is the fact that it is a \emph{semidistributive} lattice~\cite{duquenne94on}.  This means that every permutation has a canonical representation as a join of permutations; thus effectively solving the word problem for these lattices.  The members of these \emph{canonical join representations} are \emph{join-irreducible} permutations, \ie permutations with a unique descent.  

In \cite{caspard00lattice}, a property stronger than semidistributivity was established for the weak order on $\Symmetric_{n}$.  It was shown that weak order lattices are \emph{congruence uniform}, which ensures a bijective connection between join-irreducible permutations and join-irreducible \emph{lattice congruences}, \ie certain equivalence relations on $\Symmetric_{n}$ compatible with the lattice structure.

N.~Reading gave a combinatorial description of the canonical join representations in the weak order in terms of \emph{noncrossing arc diagrams}~\cite{reading15noncrossing}.  Each join-irreducible permutation of $\Symmetric_{n}$ corresponds to a unique arc connecting two distinct elements of $[n]$, and a certain \emph{forcing order} on these arcs can be used to characterize quotient lattices of the weak order.

One of these quotient lattices is the \emph{Tamari lattice}, first introduced in \cite{tamari51monoides} via a rotation transformation on binary trees.  When considered as a quotient lattice of the weak order on $\Symmetric_{n}$, the Tamari lattice---denoted by $\Tamari(n)$---arises as the subposet induced by \emph{$231$-avoiding} permutations, \ie permutations whose one-line notation does not contain a subword that normalizes to $231$~\cites{bjorner97shellable,reading06cambrian}.  

The Tamari lattices have an even richer structure than the weak order on permutations~\cite{hoissen12associahedra}.  The Tamari lattices inherit semidistributivity and congruence-uniformity from the weak order, but they are also \emph{trim}, \ie extremal and left modular~\cites{blass97moebius,markowsky92primes}.  The first property implies that their number of join-irreducible elements is as small as possible, and the second property entails some desirable topological properties.

The noncrossing arc diagrams representing the elements of $\Tamari(n)$ are precisely the \emph{noncrossing partitions} of $[n]$ introduced in \cite{kreweras72sur}; see~\cite{reading15noncrossing}.  Then, generalizing a geometric construction by N.~Reading, there is a natural way to reorder the elements of $\Tamari(n)$ which turns out to agree with the refinement order on noncrossing partitions~\cites{bancroft11shard,reading11noncrossing}.  

Let us expand on this construction a little bit.  Since the weak order is congruence uniform, we may use a perspectivity relation to label the edges in its diagram by join-irreducible permutations.  With any permutation $w$, we can associate a particular interval in the weak order by taking the meet of the elements covered by $w$.  The \emph{core label set} of $w$ is the set of labels appearing in this interval and the \emph{core label order} orders $\Symmetric_{n}$ with respect to containment of these core label sets.

Note that this construction is purely lattice-theoretic and depends only on a (finite) lattice $\Lattice$ and a labeling of the diagram of $\Lattice$.  Under certain hypotheses on this labeling, we can associate a \emph{core label order} $\CLO(\Lattice)$ with any labeled lattice.  A study of this core label order for congruence-uniform lattices was carried out in \cite{muehle19the}.  

In this article, we study a recent generalization of $\Tamari(n)$ which arise in the study of \emph{parabolic quotients} of $\Symmetric_{n}$.  Any integer composition $\alpha=(\alpha_{1},\alpha_{2},\ldots,\alpha_{r})$ of $n$ partitions the set $[n]$ into \emph{$\alpha$-regions}, \ie consecutive intervals of lengths $\alpha_{1}$, $\alpha_{2}$, $\ldots$, $\alpha_{r}$.  We then consider the set $\Symmetric_{\alpha}$ of permutations whose one-line notation---when partitioned into $\alpha$-regions---has only increasing blocks.  If $\alpha=(1,1,\ldots,1)$, this construction recovers $\Symmetric_{n}$.

The \emph{parabolic Tamari lattice} is the restriction of the weak order to the subset of $\Symmetric_{\alpha}$ consisting of those permutations avoiding certain $231$-patterns.  By \cite{muehle19tamari}*{Theorem~1}, the resulting partially ordered set, denoted by $\Tamari(\alpha)$, is a quotient lattice of the weak order on $\Symmetric_{\alpha}$.  We set out for a structural study of these lattices and certain related structures.  Our first main result establishes that $\Tamari(\alpha)$ is congruence uniform and trim.

\begin{theorem}\label{thm:alpha_tamari_structure}
	For all $n>0$ and every composition $\alpha$ of $n$, the lattice $\Tamari(\alpha)$ is congruence uniform and trim.
\end{theorem}

Since $\Tamari(\alpha)$ is congruence uniform, we may consider its core label order.  Using a modification of Reading's noncrossing arc diagrams, we may relate the core label order of $\Tamari(\alpha)$ to the refinement order on certain set partitions of $[n]$.  Exploiting the fact that $\Tamari(\alpha)$ is a quotient lattice of the weak order allows us to prove the following structural property of $\CLO\bigl(\Tamari(\alpha)\bigr)$.

\begin{theorem}\label{thm:alpha_tamari_core_label_order}
	Let $n>0$ and let $\alpha$ be a composition of $n$.  The core label order of $\Tamari(\alpha)$ is a meet-semilattice.  It is a lattice if and only if $\alpha=(n)$ or $\alpha=(1,1,\ldots,1)$.
\end{theorem}

As a consequence of Theorem~\ref{thm:alpha_tamari_structure}, $\Tamari(\alpha)$ is extremal and thus admits a canonical representation as a lattice of set pairs defined on a certain directed graph; the \emph{Galois graph}~\cites{markowsky92primes,thomas19rowmotion}.  We give a combinatorial characterization of this graph in terms of join-irreducible permutations.  We denote the (unique) join-irreducible permutation of $\Tamari(\alpha)$, whose only descent is $(a,b)$, by $w_{a,b}$.

\begin{theorem}\label{thm:alpha_tamari_galois_graph}
	Let $n>0$ and let $\alpha$ be a composition of $n$.  The Galois graph of $\Tamari(\alpha)$ is isomorphic to the directed graph whose vertices are the join-irreducible elements of $\Tamari(\alpha)$ and in which there exists a directed edge $w_{a,b}\to w_{a',b'}$ if and only if $w_{a,b}\neq w_{a',b'}$ and
	\begin{itemize}
		\item either $a$ and $a'$ belong to the same $\alpha$-region and $a\leq a'<b'\leq b$,
		\item or $a$ and $a'$ belong to different $\alpha$-regions and $a'<a<b'\leq b$, where $a$ and $b'$ belong to different $\alpha$-regions, too.
	\end{itemize}
\end{theorem}

Along the way we characterize the subposet of $\Tamari(\alpha)$ consisting of the join-irreducible permutations. 

\begin{theorem}\label{thm:alpha_tamari_irreducibles_poset}
	Let $n>0$ and let $\alpha=(\alpha_{1},\alpha_{2},\ldots,\alpha_{r})$ be a composition of $n$.  The poset of join-irreducible elements of $\Tamari(\alpha)$ consists of $r-1$ connected components, where for $j\in[r-1]$ the $j\th$ component is isomorphic to the direct product of an $\alpha_{j}$-chain and an $(\alpha_{j+1}+\alpha_{j+2}+\cdots+\alpha_{r})$-chain.
\end{theorem}

\medskip

This article is organized as follows.  In Section~\ref{sec:preliminaries}, we define the main objects considered here: (parabolic quotients of) the symmetric group, the weak order and the (parabolic) Tamari lattices.  In order to keep the combinatorial flow of this article going, we have collected the necessary order- and lattice-theoretic concepts in Appendix~\ref{sec:posets_lattices}.  We recommend to read the combinatorial parts of this article in order and refer to the appendix whenever unknown terminology is encountered.

In Section~\ref{sec:alpha_tamari_congruence_uniform}, we prove that the parabolic Tamari lattices are congruence uniform and study their associated core label order.  We investigate the join-irreducible elements in the parabolic Tamari lattices in Section~\ref{sec:alpha_tamari_trim} and prove our main results.  We conclude this article with an enumerative observation relating the generating function of the M{\"o}bius function in the core label order of the parabolic Tamari lattices and the generating function of antichains in certain partially ordered sets in Section~\ref{sec:alpha_chapoton}.

\section{Preliminaries}
	\label{sec:preliminaries}
\subsection{The symmetric group and the weak order}
	\label{sec:symmetric_group}
For $n>0$, the \defn{symmetric group} $\Symmetric_{n}$ is the group of permutations of $[n]\defs\{1,2,\ldots,n\}$ under composition.  For $w\in\Symmetric_{n}$ and $i\in[n]$ we write $w_{i}$ instead of $w(i)$.  The \defn{one-line notation} of $w$ is the string $\oneline{w_{1}\;w_{2}\;\ldots\;w_{n}}$.

For $i,j\in[n]$ with $i<j$, the permutation that exchanges $i$ and $j$ and fixes everything else is a \defn{transposition}; denoted by $t_{i,j}$.  If $j=i+1$, then we write $s_{i}$ instead of $t_{i,i+1}$.  

The one-line notation of $w\circ t_{i,j}$ is the same as the one-line notation of $w$ except that the $i\th$ and the $j\th$ entries are swapped.  The one-line notation of $t_{i,j}\circ w$ is the same as the one-line notation of $w$ except that the positions of the values $i$ and $j$ are swapped.

A \defn{(right) inversion} of $w$ is a pair $(i,j)$ with $i<j$ and $w_{i}>w_{j}$.  A \defn{(right) descent} of $w$ is a pair $(i,j)$ with $i<j$ and $w_{i}=w_{j}+1$.  Let $\Inv(w)$ denote the set of (right) inversions of $w$, and let $\Des(w)$ denote the set of (right) descents of $w$.

\begin{remark}
	We wish to emphasize that we consider ``inversions'' with respect to \emph{positions}, meaning that composing on the \emph{right} with a transposition swaps the entries in \emph{positions} $i$ and $j$.  
	
	It is much more common in Coxeter-Catalan theory to consider \emph{left} inversions with respect to \emph{values}.  The reason for choosing this convention is the fact that it is more convenient for us to spot membership in parabolic quotients this way.
\end{remark}

Ordering permutations of $[n]$ with respect to containment of their (right) inversion sets yields the \defn{(left) weak order}, denoted by $\leq_{L}$.  For any subset $X\subseteq\Symmetric_{n}$ we write $\Weak(X)\defs(X,\leq_{L})$ for the set $X$ partially ordered by $\leq_{L}$.  Figure~\ref{fig:weak_order_s4} shows $\Weak(\Symmetric_{4})$.  

\begin{figure}
	\centering
	\begin{subfigure}[t]{.45\textwidth}
		\centering
		\includegraphics[scale=.7,page=1]{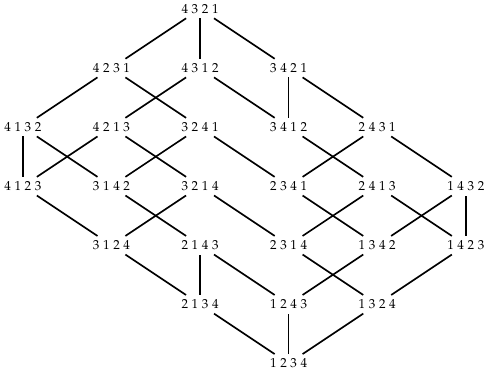}
		\caption{The lattice $\Weak(\Symmetric_{4})$.}
		\label{fig:weak_order_s4}
	\end{subfigure}
	\hspace*{1cm}
	\begin{subfigure}[t]{.45\textwidth}
		\centering
		\includegraphics[scale=.7,page=2]{figures_aoco.pdf}
		\caption{The lattice $\Weak\bigl(\Symmetric_{4}(231)\bigr)$.}
		\label{fig:tamari_s4}
	\end{subfigure}
	\caption{Two lattices of permutations.}
\end{figure}

It follows from definition of the (left) weak order that two permutations $u,v$ satisfy $u\lessdot_{L}v$ if and only if $\Inv(v)\setminus\Inv(u)=\bigl\{(i,j)\bigr\}$ and $v_{i}=v_{j}+1$.  In other words, $u\lessdot_{L}v$ if and only if $v=s_{u_{i}}\circ u$ and $v$ has more inversions than $u$.

\begin{theorem}[\cites{guilbaud71analyse,yanagimoto69partial,caspard00lattice}]\label{thm:weak_order_lattice}
	For all $n>0$, $\Weak(\Symmetric_{n})$ is a congruence-uniform lattice.
\end{theorem}

See Section~\ref{sec:congruence_uniform_lattices} for the definition of a congruence-uniform lattice.  A consequence of Theorem~\ref{thm:weak_order_lattice} is the existence of a least element (the \defn{identity} $\id\defs\oneline{1\;2\;\ldots\;n}$) and a greatest element (the \defn{long element} $\wo\defs\oneline{n\;n{-}1\;\ldots\;1}$) in $\Weak(\Symmetric_{n})$.

\subsection{$231$-avoiding permutations and the Tamari lattice}
	\label{sec:tamari_lattice}
We now exhibit an important sub- and quotient lattice of $\Weak(\Symmetric_{n})$, the \defn{Tamari lattice} $\Tamari(n)$.  

A \defn{$231$-pattern} in a permutation $w\in\Symmetric_{n}$ is a triple $(i,j,k)$ with $i<j<k$ and $w_{k}<w_{i}<w_{j}$.  Then, $w$ is \defn{$231$-avoiding} if it does not have a $231$-pattern.  Let $\Symmetric_{n}(231)$ denote the set of $231$-avoiding permutations of $[n]$.  

The Tamari lattice is, as far as we are concerned, the poset 
\begin{displaymath}
	\Tamari(n)\defs\Weak\bigl(\Symmetric_{n}(231)\bigr).
\end{displaymath}
This poset is named after D.~Tamari, who introduced it in \cite{tamari51monoides} via a partial order on binary trees and proved its lattice property.  The fact that $\Weak\bigl(\Symmetric_{n}(231)\bigr)$ incarnates $\Tamari(n)$ follows from \cite{bjorner97shellable}*{Theorem~9.6}.  Figure~\ref{fig:tamari_s4} shows $\Tamari(4)$.

The next theorem states some important, lattice-theoretic properties of $\Tamari(n)$.  See Sections~\ref{sec:congruence_uniform_lattices} and \ref{sec:trim_lattices} for the corresponding definitions.

\begin{theorem}[\cites{bjorner97shellable,blass97moebius,geyer94on,markowsky92primes,reading06cambrian}]\label{thm:tamari_properties}
	For all $n>0$, $\Tamari(n)$ is a sublattice and a quotient lattice of $\Weak(\Symmetric_{n})$.  Moreover, $\Tamari(n)$ is trim and congruence uniform.
\end{theorem}

\subsection{Parabolic quotients of the symmetric group}
	\label{sec:parabolic_quotients_symmetric_group}
Let $\alpha=(\alpha_{1},\alpha_{2},\ldots,\alpha_{r})$ be a composition of $n$, and define $p_{0}\defs 0$ and $p_{i}\defs\alpha_{1}+\alpha_{2}+\cdots+\alpha_{i}$ for $i\in[r]$.  We define the \defn{parabolic subgroup} of $\Symmetric_{n}$ with respect to $\alpha$ by 
\begin{displaymath}
	G_{\alpha} \defs \Symmetric_{\alpha_{1}}\times\Symmetric_{\alpha_{2}}\times\cdots\times\Symmetric_{\alpha_{r}}.
\end{displaymath}

The symmetric group $\Symmetric_{n}$ is generated by its set $\{s_{1},s_{2},\ldots,s_{n-1}\}$ of adjacent transpositions, and therefore any $w\in\Symmetric_{n}$ can be written as a product of the $s_{i}$'s. The \defn{length} of $w$ is the minimal number of adjacent transpositions needed to form $w$ as such a product.  It follows from \cite{yanagimoto69partial}*{Proposition~2.1} that the length of $w$ equals its number of inversions.

Let $\Symmetric_{\alpha}$ be the set of all minimal-length representatives of the left cosets of $G_{\alpha}$ in $\Symmetric_{n}$, \ie
\begin{align*}
	\Symmetric_{\alpha} & \defs \Bigl\{w\in\Symmetric_{n}\mid \bigl\rvert\Inv(w)\bigr\rvert<\bigl\lvert\Inv(ws_{i})\bigr\rvert\;\text{for}\;i\notin\{p_{1},p_{2},\ldots,p_{r-1}\}\Bigr\}\\
	& = \Bigl\{w\in\Symmetric_{n}\mid w_{i}<w_{i+1}\;\text{for}\;i\notin\{p_{1},p_{2},\ldots,p_{r-1}\}\Bigr\}.
\end{align*}
The elements of $\Symmetric_{\alpha}$ are \defn{$\alpha$-permutations}.  

For $i\in[r]$, the set $\bigl\{p_{i-1}{+}1,p_{i-1}{+}2,\ldots,p_{i}\bigr\}$ is the \defn{$i\th$ $\alpha$-region}.  We indicate $\alpha$ in the one-line notation of $w\in\Symmetric_{n}$ either by coloring every $\alpha$-region with a different color or by separating $\alpha$-regions by a vertical bar.  Figure~\ref{fig:weak_order_121} highlights the elements of $\Symmetric_{(1,2,1)}$ in $\Weak(\Symmetric_{4})$.    

\begin{figure}
	\centering
	\begin{subfigure}[t]{.45\textwidth}
		\centering
		\includegraphics[scale=.7,page=3]{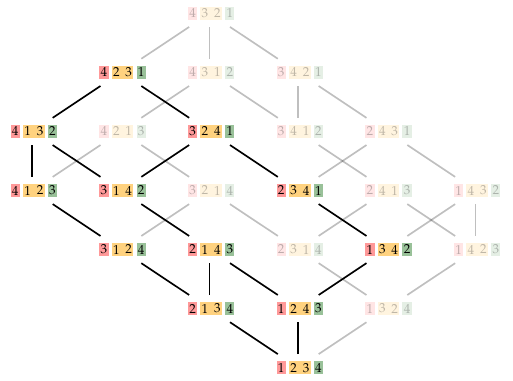}
		\caption{$\Weak(\Symmetric_{(1,2,1)})$ as an interval of $\Weak(\Symmetric_{4})$.}
		\label{fig:weak_order_121}
	\end{subfigure}
	\hspace*{1cm}
	\begin{subfigure}[t]{.45\textwidth}
		\centering
		\includegraphics[scale=.7,page=4]{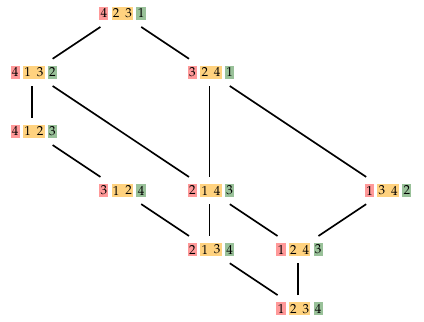}
		\caption{The lattice $\Weak\bigl(\Symmetric_{(1,2,1)}(231)\bigr)$.}
		\label{fig:tamari_121}
	\end{subfigure}
	\caption{Two lattices of $(1,2,1)$-permutations.}
	\label{fig:alpha_lattices}
\end{figure}

The set $\Symmetric_{\alpha}$ behaves quite well with respect to left weak order.

\begin{theorem}[\cite{bjorner88generalized}]\label{thm:alpha_weak_order_lattice}
	For all $n>0$ and every composition $\alpha$ of $n$, $\Weak(\Symmetric_{\alpha})$ is a principal order ideal in $\Weak(\Symmetric_{n})$.  Consequently, $\Weak(\Symmetric_{\alpha})$ is a congruence-uniform lattice.
\end{theorem}

A consequence of Theorem~\ref{thm:alpha_weak_order_lattice} is the existence of a greatest element in $\Weak(\Symmetric_{\alpha})$.  This element is denoted by $\woa$ and its one-line notation is of the following form:
\begin{displaymath}
	\oneline{\underbrace{n{-}p_{1}{+}1,n{-}p_{1}{+}2,\ldots,n}_{\alpha_{1}} \mid \underbrace{n{-}p_{2}{+}1,n{-}p_{2}{+}2,\ldots,n{-}p_{1}}_{\alpha_{2}}\mid \ldots \mid \underbrace{1,2,\ldots,n{-}p_{r-1}}_{\alpha_{r}}}.
\end{displaymath}
The vertical bars have no impact on the one-line notation, they shall only help separating the $\alpha$-regions.

Clearly, if $\alpha=(1,1,\ldots,1)$ is a composition of $n$, then $\Symmetric_{(1,1,\ldots,1)}=\Symmetric_{n}$.

\subsection{Parabolic $231$-avoiding permutations and the parabolic Tamari lattice}
	\label{sec:alpha_tamari_lattice}
Generalizing the constructions from Section~\ref{sec:tamari_lattice}, we now identify a particular quotient lattice of $\Weak(\Symmetric_{\alpha})$.  

An \defn{($\alpha,231)$-pattern} in an $\alpha$-permutation $w\in\Symmetric_{\alpha}$ is a triple $(i,j,k)$ with $i<j<k$ all in different $\alpha$-regions such that $w_{k}<w_{i}<w_{j}$ and $w_{i}=w_{k}+1$.  Then, $w$ is \defn{$(\alpha,231)$-avoiding} if it does not have an $(\alpha,231)$-pattern.  Let $\Symmetric_{\alpha}(231)$ denote the set of $(\alpha,231)$-avoiding permutations of $[n]$.

The \defn{$\alpha$-Tamari lattice} is the poset $\Tamari(\alpha)\defs\Weak\bigl(\Symmetric_{\alpha}(231)\bigr)$.  This name is justified by the following result.

\begin{theorem}[\cite{muehle19tamari}*{Theorem~1}]\label{thm:alpha_tamari_lattice}
	For all $n>0$ and every composition $\alpha$ of $n$, $\Tamari(\alpha)$ is a quotient lattice of $\Weak(\Symmetric_{\alpha})$.  
\end{theorem}

Figure~\ref{fig:tamari_121} shows the $(1,2,1)$-Tamari lattice.  We can witness in this example that $\Tamari(\alpha)$ is in general \emph{not} a sublattice of $\Weak(\Symmetric_{\alpha})$, since the permutations $\oneline{4\mid 1\;3\mid 2}$ and $\oneline{3\mid 2\;4\mid 1}$ have a different meet in $\Weak(\Symmetric_{(1,2,1)})$ than in $\Tamari\bigl((1,2,1)\bigr)$.  

Since $\Tamari(\alpha)$ is a quotient lattice of $\Weak(\Symmetric_{\alpha})$, there exists a surjective lattice map $\proj_{\alpha}\colon\Symmetric_{\alpha}\to\Symmetric_{\alpha}(231)$ which maps $w\in\Symmetric_{\alpha}$ to the greatest $(\alpha,231)$-avoiding permutation below $w$ in weak order~\cite{muehle19tamari}*{Lemma~12}.

\begin{remark}
	In general, we need to distinguish cover relations in $\Tamari(\alpha)$ from cover relations in $\Weak(\Symmetric_{\alpha})$, and we do so by using $\lessdot_{\alpha}$ (resp. $\lessdot_{L}$).  The reason is that for $u,v\in\Symmetric_{\alpha}$, $u\lessdot_{\alpha}v$ does not necessarily imply $u\lessdot_{L}v$; see Figure~\ref{fig:alpha_lattices}.  More generally, we indicate poset- and lattice-theoretic notions in $\Tamari(\alpha)$ with a subscript ``$\alpha$'', and in $\Weak(\Symmetric_{\alpha})$ with a subscript ``$L$''.
\end{remark}

Again, if $\alpha=(1,1,\ldots,1)$ is a composition of $n$, then it follows from \cite{muehle19tamari}*{Proposition~8} that $\Tamari\bigl((1,1,\ldots,1)\bigr)=\Tamari(n)$.  
In the remainder of this article, we study the $\alpha$-Tamari lattices from a lattice-theoretic and combinatorial perspective.

\section{The $\alpha$-Tamari lattices are congruence-uniform}
	\label{sec:alpha_tamari_congruence_uniform}
We start right away with the proof that $\Tamari(\alpha)$ is congruence uniform.

\begin{proposition}\label{prop:alpha_tamari_congruence_uniform}
	For all $n>0$ and every composition $\alpha$ of $n$, $\Tamari(\alpha)$ is congruence uniform. 
\end{proposition}
\begin{proof}
	By Theorem~\ref{thm:weak_order_lattice}, $\Weak(\Symmetric_{n})$ is congruence uniform and by Theorems~\ref{thm:alpha_weak_order_lattice} and \ref{thm:alpha_tamari_lattice}, $\Tamari(\alpha)$ is a quotient lattice of an interval of $\Weak(\Symmetric_{n})$.  By \cite{day79characterizations}*{Theorem~4.3}, congruence-uniformity is preserved under passing to sublattices and quotient lattices.  This proves the claim.
\end{proof}

\begin{corollary}\label{cor:alpha_tamari_semidistributive}
	For all $n>0$ and every composition $\alpha$ of $n$, $\Tamari(\alpha)$ is semidistributive.
\end{corollary}
\begin{proof}
	This follows from Proposition~\ref{prop:alpha_tamari_congruence_uniform} and Theorem~\ref{thm:congruence_uniform_semidistributive}.
\end{proof}

\subsection{Noncrossing $\alpha$-partitions}
	\label{sec:alpha_partitions}
Our next goal is a combinatorial description of the canonical join representations in $\Tamari(\alpha)$.  In preparation, we introduce another combinatorial family parametrized by $\alpha$.

An \defn{$\alpha$-arc} is a pair $(a,b)$, where $1\leq a<b\leq n$ and $a,b$ belong to different $\alpha$-regions.  Two $\alpha$-arcs $(a_{1},b_{1})$ and $(a_{2},b_{2})$ are \defn{compatible} if $a_{1}\neq a_{2}$ or $b_{1}\neq b_{2}$, and the following is satisfied:
\begin{description}
	\item[NC1\label{it:nc1}] if $a_{1}<a_{2}<b_{1}<b_{2}$, then either $a_{1}$ and $a_{2}$ lie in the same $\alpha$-region or $b_{1}$ and $a_{2}$ lie in the same $\alpha$-region;
	\item[NC2\label{it:nc2}] if $a_{1}<a_{2}<b_{2}<b_{1}$, then $a_{1}$ and $a_{2}$ lie in different $\alpha$-regions.
\end{description}

An \defn{$\alpha$-partition} is a set partition of $[n]$, where no block intersects an $\alpha$-region in more than one element.  Let $\Pi_{\alpha}$ denote the set of $\alpha$-partitions of $[n]$.  

Let $\Pf\in\Pi_{\alpha}$, and let $B\in\Pf$ be a block.  If $a,b\in B$, then we write $a\sim_{\Pf}b$.  A \defn{bump} of $\Pf$ is a pair $(a,b)$ such that $a,b\in B$ and there is no $c\in B$ with $a<c<b$.  

Clearly, any bump of $\Pf$ is an $\alpha$-arc.  An $\alpha$-partition is \defn{noncrossing} if its bumps are pairwise compatible $\alpha$-arcs.  We denote the set of all noncrossing $\alpha$-partitions by $\NC(\alpha)$.  We use the term \emph{parabolic noncrossing partitions} to refer to noncrossing $\alpha$-partitions for unspecific $\alpha$.

\begin{remark}
	If $\alpha=(1,1,\ldots,1)$, then every $\alpha$-region is a singleton, so that \eqref{it:nc2} will always be satisfied and \eqref{it:nc1} can never be satisfied.  Thus, the noncrossing $(1,1,\ldots,1)$-partitions are precisely the set partitions of $[n]$ without any two bumps $(a_{1},b_{1})$ and $(a_{2},b_{2})$ for $a_{1}<a_{2}<b_{1}<b_{2}$.  These \emph{ordinary noncrossing partitions} were introduced in \cite{kreweras72sur} and have been a frequent object of study ever since.
\end{remark}

We graphically represent an $\alpha$-arc $(a,b)$ as follows. We draw $n$ nodes on a horizontal line, label them by $1,2,\ldots,n$ from left to right, and group them together according to $\alpha$-regions.  Now we draw a curve leaving the node labeled $a$ to the bottom, staying below the $\alpha$-region containing $a$, moving up and over the subsequent $\alpha$-regions until it enters the node labeled $b$ from above.  

An $\alpha$-partition is noncrossing if and only if its bumps can be drawn in this manner such that no two curves intersect in their interior.  Likewise, any collection of pairwise compatible $\alpha$-arcs corresponds to a noncrossing $\alpha$-partition whose blocks are given by the connected components of the graphical representation of the $\alpha$-arcs.  Figure~\ref{fig:alpha_partition} shows a graphical representation of a noncrossing $\alpha$-partition.

\begin{figure}
	\centering
	\begin{subfigure}[t]{1\textwidth}
		\centering
		\includegraphics[scale=1,page=5]{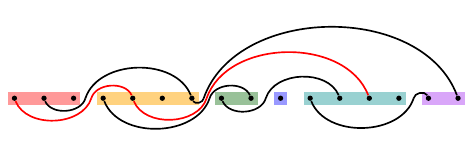}
		\caption{$\Pf=\bigl\{\{1,5,13\},\{2,7,16\},\{3\},\{4,9\},\{6\},\{8,12\},\{10\},\{11,15\},\{14\}\bigr\}\in\NC(\alpha)$.}
		\label{fig:alpha_partition}
	\end{subfigure}
	
	\vspace*{.5cm}

	\begin{subfigure}[t]{.3\textwidth}
		\centering
		\includegraphics[scale=.75,page=6]{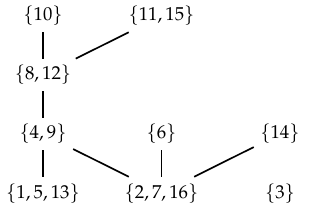}
		\caption{The poset $\vec{O}_{\Pf}$.}
		\label{fig:alpha_block_order}
	\end{subfigure}
	\hspace*{1cm}
	\begin{subfigure}[t]{.6\textwidth}
		\centering
		\includegraphics[scale=.75,page=7]{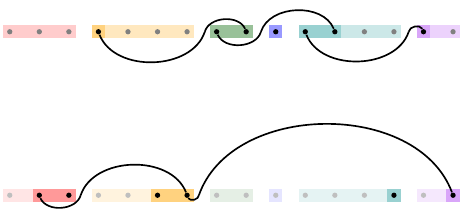}
		\caption{The smaller parabolic noncrossing partitions $\Pf_{1}\in\NC\bigl((1,2,1,2,1)\bigr)$ (top) and $\Pf_{2}\in\NC\bigl((2,2,1,1)\bigr)$ (bottom).}
		\label{fig:alpha_partition_recursion}
	\end{subfigure}

	\vspace*{.5cm}
	
	\begin{subfigure}[t]{.8\textwidth}
		\centering
		\includegraphics[scale=1,page=8]{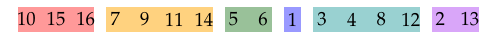}
		\caption{The $(\alpha,231)$-avoiding permutation $\BijNCAlign_{\alpha}^{-1}(\Pf)$.}
		\label{fig:alpha_permutation}
	\end{subfigure}
	\caption{An illustration of Theorem~\ref{thm:nc_231_bijection} for $\alpha=(3,4,2,1,4,2)$.}
	\label{fig:nc_bijection}
\end{figure}

\begin{theorem}[\cite{muehle19tamari}*{Theorem~4.1}]\label{thm:nc_231_bijection}
	For all $n>0$ and every composition $\alpha$ of $n$, the sets $\Symmetric_{\alpha}(231)$ and $\NC(\alpha)$ are in bijection.  This bijection sends descents to bumps.
\end{theorem}

Let $\BijNCAlign_{\alpha}$ denote the bijection from Theorem~\ref{thm:nc_231_bijection}.  If $w\in\Symmetric_{\alpha}(231)$, then $\Des(w)$ is a collection of pairwise compatible $\alpha$-arcs; and thus corresponds to some $\BijNCAlign_{\alpha}(w)\in\NC(\alpha)$.  Conversely, if $\Pf\in\NC(\alpha)$, then we define an acyclic binary relation $\vec{R}_{\Pf}$ on the blocks of $\Pf$ by setting $(B,B')\in\vec{R}_{\Pf}$ if and only if there exists an $\alpha$-arc $(a,b)$ such that $a,b\in B$ and $a<\min B'<b$ for $a$ and $\min B'$ in different $\alpha$-regions.  

Let $\vec{O}_{\Pf}$ denote the reflexive and transitive closure of $\vec{R}_{\Pf}$.  Without loss of generality, we may assume that $B_{1}=\{i_{1},i_{2},\ldots,i_{k}\}\in\Pf$ with $1=i_{1}<i_{2}<\cdots<i_{k}$.  Then, $B_{1}$ is minimal in $\vec{O}_{\Pf}$.  We construct a permutation $w=\BijNCAlign_{\alpha}^{-1}(\Pf)\in\Symmetric_{\alpha}(231)$ inductively by setting $w_{i_{j+1}}=w_{i_{j}}-1$ for $j\in[k-1]$, and $w_{1}=\lvert X\rvert$, where $X$ is the union of the blocks in the order filter of $\vec{O}_{\Pf}$ generated by $B_{1}$.  The remaining values for $w$ are determined by considering two smaller parabolic noncrossing partitions $\Pf_{1}$ and $\Pf_{2}$, where $\Pf_{1}$ is the restriction of $\Pf$ to $X\setminus B_{1}$, and where $\Pf_{2}$ is the restriction of $\Pf$ to $[n]\setminus X$.  Note that $\vec{O}_{\Pf_{1}}$ and $\vec{O}_{\Pf_{2}}$ are induced subposets of $\vec{O}_{\Pf}$.  See Figure~\ref{fig:nc_bijection} for an illustration.

\subsection{Canonical join representations in $\Tamari(\alpha)$}
	\label{sec:alpha_join_representations}
We now explain how to use noncrossing $\alpha$-partitions to describe canonical join representations in $\Tamari(\alpha)$.  Essentially, we are going to prove that, for $w\in\Symmetric_{\alpha}(231)$, the set of bumps of $\BijNCAlign_{\alpha}(w)$ determines the canonical join representation of $w$ in $\Tamari(\alpha)$.

\begin{proposition}\label{prop:alpha_join_representations}
	For all $n>0$ and every composition $\alpha$ of $n$, the canonical join representation of $w\in\Symmetric_{\alpha}(231)$ in $\Tamari(\alpha)$ is $\bigl\{w_{a,b}\mid (a,b)\in\Des(w)\bigr\}$.
\end{proposition}

We now gather some ingredients required for the proof of Proposition~\ref{prop:alpha_join_representations}.

\begin{lemma}\label{lem:descents_covers}
	For $w\in\Symmetric_{\alpha}(231)$, the number of descents of $w$ equals the number of elements of $\Tamari(\alpha)$ covered by $w$.
\end{lemma}
\begin{proof}
	Let $w\in\Symmetric_{\alpha}(231)$, and let $n_{L}$ (resp. $n_{\alpha}$) denote the number of elements of $\Weak(\Symmetric_{\alpha})$ (resp. $\Tamari(\alpha))$ covered by $w$.

	By definition of the weak order, $n_{L}=\bigl\lvert\Des(w)\bigr\rvert$.  Since $\Weak(\Symmetric_{\alpha})$ is congruence uniform, $n_{L}$ equals the number of canonical joinands of $w$ in $\Weak(\Symmetric_{\alpha})$ by Corollary~\ref{cor:canonical_joinands_lower_covers}.  If $w\in\Symmetric_{\alpha}(231)$, then $\proj_{\alpha}(w)=w$ and Proposition~\ref{prop:canonical_joinands_quotient} implies that $n_{L}$ is the number of canonical joinands of $w$ in $\Tamari(\alpha)$, which is also $n_{\alpha}$.
\end{proof}

\begin{corollary}\label{cor:arcs_join_irreducibles}
	The set of $\alpha$-arcs is in bijection with the set $\JI\bigl(\Tamari(\alpha)\bigr)$ of join-irreducible elements of $\Tamari(\alpha)$.  
\end{corollary}
\begin{proof}
	An element $w\in\Tamari(\alpha)$ is join irreducible if and only if it covers a unique element.  By Lemma~\ref{lem:descents_covers}, $w$ is join irreducible if and only if it has a unique descent. By Theorem~\ref{thm:nc_231_bijection}, $\BijNCAlign_{\alpha}(w)$ is a noncrossing $\alpha$-partition with a unique bump, and thus corresponds to an $\alpha$-arc.  Since $\BijNCAlign_{\alpha}$ is a bijection, the claim follows.
\end{proof}

\begin{corollary}\label{cor:alpha_irreducibles}
	Let $(a,b)$ be an $\alpha$-arc, where $a$ belongs to the $j\th$ $\alpha$-region.  The corresponding join-irreducible element of $\Tamari(\alpha)$ is $w_{a,b}\in\Symmetric_{\alpha}(231)$ given by
	\begin{displaymath}
		w_{a,b}(i) = \begin{cases}i, & \text{if}\;i<a\;or\;i>b,\\ a+b-p_{j}+k, & \text{if}\;i=a+k\;\text{for}\;0\leq k\leq p_{j}-a,\\ a+k-1, & \text{if}\;i=p_{j}+k\;\text{for}\;1\leq k\leq b-p_{j}.\end{cases}
	\end{displaymath}
	The inversion set of $w_{a,b}$ is
	\begin{displaymath}
		\Inv(w_{a,b}) = \bigl\{(k,l)\mid a\leq k\leq p_{j}, p_{j}+1\leq l\leq b\bigr\}.
	\end{displaymath}
\end{corollary}

Corollary~\ref{cor:alpha_irreducibles} implies that the inversion set of $w_{a,b}$ can be read off easily from $\BijNCAlign(w_{a,b})$.  In fact, the first components of an inversion of $w_{a,b}$ are the nodes that lie weakly to the right of $a$ and weakly above the arc connecting nodes $a$ and $b$, the second components are the nodes that lie weakly below this arc.  This is illustrated in the following example in the case $n=8$, $a=2$, $b=6$; see also Figure~\ref{fig:alpha_irreducibles_inversions}.

\begin{figure}
	\centering
	\includegraphics[scale=1,page=33]{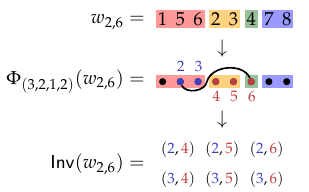}
	\caption{Illustrating the inversion set of a join-irreducible $(\alpha,231)$-avoiding permutation.}
	\label{fig:alpha_irreducibles_inversions}
\end{figure}

\begin{example}
	Let $\alpha=(3,2,1,2)$.  The join-irreducible permutation $w_{2,6}\in\Symmetric_{\alpha}(231)$ is given by the one-line notation $\oneline{1\;5\;6\mid 2\;3\mid 4\mid 7\;8}$.  Its inversion set is
	\begin{displaymath}
		\Inv(w_{2,6}) = \bigl\{(2,4),(2,5),(2,6),(3,4),(3,5),(3,6)\bigr\}.
	\end{displaymath}
\end{example}

\begin{lemma}\label{lem:alpha_cover_labels}
	Let $u,v\in\Symmetric_{\alpha}(231)$ with $u\lessdot_{\alpha}v$.  There exists a unique $(a,b)\in\Des(v)$ such that $(a,b)\notin\Inv(u)$.
\end{lemma}
\begin{proof}
	Let $v\in\Symmetric_{\alpha}(231)$.  By Lemma~\ref{lem:descents_covers}, the number of permutations $u\in\Symmetric_{\alpha}(231)$ with $u\lessdot_{\alpha} v$ equals $\bigl\lvert\Des(v)\bigr\rvert$.  Thus, for every $(a,b)\in\Des(v)$ there is a unique $u\in\Symmetric_{\alpha}(231)$ with $u\lessdot_{\alpha}v$.  It remains to show that $(a,b)\notin\Inv(u)$.
	
	The permutation $u_{1}=s_{u_{a}}\circ v\in\Symmetric_{\alpha}$---in whose one-line notation the entries in positions $a$ and $b$ are swapped---satisfies $u_{1}\lessdot_{L}v$, and it follows $(a,b)\notin\Inv(u_{1})$.  Now, consider $u=\proj_{\alpha}(u_{1})\in\Symmetric_{\alpha}(231)$.  By construction, $u\leq_{L}u_{1}$, which means $\Inv(u)\subseteq\Inv(u_{1})$.  Thus, $(a,b)\notin\Inv(u)$.
\end{proof}

Recall the definition of perspective cover relations from Section~\ref{sec:congruence_uniform_lattices}.

\begin{proposition}\label{prop:alpha_cu_labeling}
	Let $u,v\in\Symmetric_{\alpha}(231)$ with $u\lessdot_{\alpha}v$, and let $(a,b)\in\Des(v)$ with $(a,b)\notin\Inv(u)$.  Then, $u\lessdot_{\alpha}v$ and ${w_{a,b}}_{*}\lessdot_{\alpha}w_{a,b}$ are perspective cover relations in $\Tamari(\alpha)$.
\end{proposition}
\begin{proof}
	Let $v\in\Symmetric_{\alpha}(231)$ and let $(a,b)\in\Des(v)$.  By Lemma~\ref{lem:alpha_cover_labels}, there is a unique $u\in\Symmetric_{\alpha}(231)$ with the desired properties.
	
	Suppose that $a$ is in the $j\th$ $\alpha$-region.  Since $v\in\Symmetric_{\alpha}(231)$, $v_{c}\leq v_{b}$ for all $c\in\{p_{j}{+}1,p_{j}{+}2,\ldots,b\}$.  By Corollary~\ref{cor:alpha_irreducibles}, $\Inv(w_{a,b})\subseteq\Inv(v)$ and thus $w_{a,b}\leq_{L}v$.
	
	Let $u_{1}=s_{u_{a}}\circ v\in\Symmetric_{\alpha}$.  Then $u_{1}\lessdot_{L}v$ and $u=\proj_{\alpha}(u_{1})$.  Then, $\Inv(u_{1})=\Inv(v)\setminus\bigl\{(a,b)\bigr\}$ and $\Inv({w_{a,b}}_{*})=\Inv(w_{a,b})\setminus\bigl\{(a,b)\bigr\}$.  In particular, ${w_{a,b}}_{*}\in\Symmetric_{\alpha}(231)$, and since $\proj_{\alpha}$ is a lattice map we conclude 
	\begin{align*}
		u\wedge_{\alpha}w_{a,b} & = \proj_{\alpha}(u_{1})\wedge_{\alpha}\proj_{\alpha}(w_{a,b}) = \proj_{\alpha}(u_{1}\wedge_{L}w_{a,b}) = \proj_{\alpha}({w_{a,b}}_{*}) = {w_{a,b}}_{*},\\
		u\vee_{\alpha}w_{a,b} & = \proj_{\alpha}(u_{1})\vee_{\alpha}\proj_{\alpha}(w_{a,b}) = \proj_{\alpha}(u_{1}\vee_{L}w_{a,b}) = \proj_{\alpha}(v) = v.
	\end{align*}
	By definition, $({w_{a,b}}_{*},w_{a,b})\perspective(u,v)$ in $\Tamari(\alpha)$.
\end{proof}

\begin{proof}[Proof of Proposition~\ref{prop:alpha_join_representations}]
	This follows from Proposition~\ref{prop:alpha_cu_labeling} using Lemma~\ref{lem:perspective_covers_same_label} and Theorem~\ref{thm:join_representation_lower_covers}.
\end{proof}

For $\alpha=(1,1,\ldots,1)$, Proposition~\ref{prop:alpha_join_representations} was previously found in \cite{reading07clusters}*{Example~6.3}.  In fact, since Theorem~\ref{thm:alpha_tamari_lattice} states that $\Tamari(\alpha)$ is a quotient lattice of $\Weak(\Symmetric_{\alpha})$, Proposition~\ref{prop:alpha_join_representations} follows immediately from Proposition~\ref{prop:canonical_joinands_quotient} in conjunction with \cite{reading11sortable}*{Theorem~8.1}.  However, because we need an explicit description of the join-irreducible $(\alpha,231)$-avoiding permutations later, we have decided to add some more details.

Moreover, for $\alpha=(1,1,\ldots,1)$, the canonical join complex of $\Tamari(\alpha)$ (see Section~\ref{sec:semidistributive_lattices}) was studied in \cite{barnard20tamari}.  In particular, it was shown in \cite{barnard20tamari}*{Theorem~1.3} that this complex is \emph{vertex decomposable}, a strong topological property introduced in \cite{provan80decompositions} which implies that this complex is homotopic to a wedge of spheres, shellable and Cohen-Macaulay.  We plan to investigate the canonical join complex of $\Tamari(\alpha)$ for arbitrary $\alpha$ in a follow-up article.  For the time being, we pose the following conjecture.

\begin{conjecture}\label{conj:alpha_join_complex_vertex_decomposable}
	For all $n>0$ and every composition $\alpha$ of $n$, the canonical join complex of $\Tamari(\alpha)$ is vertex decomposable.
\end{conjecture}

\subsection{The core label order of $\Tamari(\alpha)$}
	\label{sec:alpha_core_label_order}
In this section we study the core label order of $\Tamari(\alpha)$, see Section~\ref{sec:core_label_order}.  By Proposition~\ref{prop:alpha_tamari_congruence_uniform}, $\Tamari(\alpha)$ is congruence uniform, and thus admits an edge-labeling with join-irreducible $(\alpha,231)$-avoiding permutations, which is determined by the perspectivity relation; see Lemma~\ref{lem:perspective_covers_same_label} and Proposition~\ref{prop:alpha_cu_labeling}.

The core label order of $\Tamari(\alpha)$ orders the elements of $\Symmetric_{\alpha}(231)$ with respect to this labeling.  By Corollary~\ref{cor:arcs_join_irreducibles}, the elements of $\JI\bigl(\Tamari(\alpha)\bigr)$ correspond bijectively to $\alpha$-arcs.  Therefore, we may identify the core label set of $w\in\Symmetric_{\alpha}(231)$ with a collection of $\alpha$-arcs.

\begin{example}
	Let $\alpha=(1,2,1)$.  Figures~\ref{fig:symmetric_4_labeled} and \ref{fig:tamari_121_labeled} show the lattices $\Weak(\Symmetric_{\alpha})$ and $\Tamari(\alpha)$, where the edges are labeled by \eqref{eq:cu_labeling}.  In Figure~\ref{fig:tamari_121_labeled}, the nodes are additionally labeled by noncrossing $\alpha$-partitions.  Figures~\ref{fig:symmetric_4_clo} and \ref{fig:tamari_121_clo} show the corresponding core label orders.
\end{example}

\begin{figure}
	\centering
	\begin{subfigure}[t]{1\textwidth}
		\centering
		\includegraphics[scale=.9,page=16]{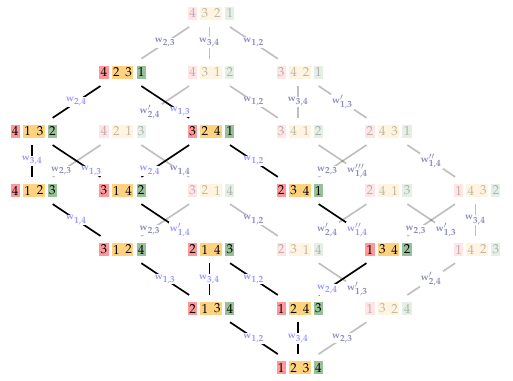}
		\caption{The lattice $\Weak(\Symmetric_{4})$ labeled by \eqref{eq:cu_labeling}, with the interval $\Weak\bigl(\Symmetric_{(1,2,1)}\bigr)$ highlighted.}
		\label{fig:symmetric_4_labeled}
	\end{subfigure}

	\vspace*{.5cm}
	
	\begin{subfigure}[t]{1\textwidth}
		\centering
		\includegraphics[width=\textwidth,page=17]{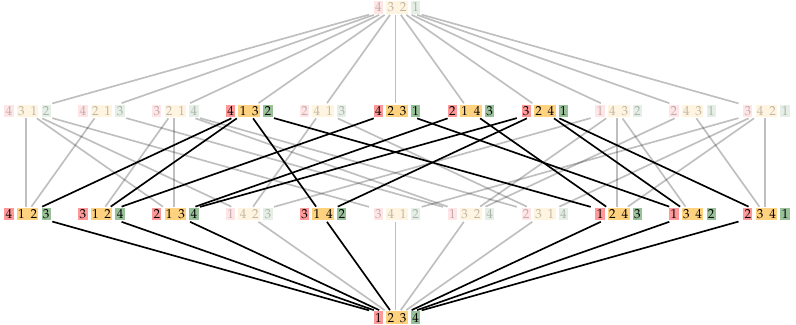}
		\caption{The lattice $\CLO\bigl(\Weak(\Symmetric_{4})\bigr)$ with the order ideal $\CLO\bigl(\Weak(\Symmetric_{(1,2,1)})\bigr)$ highlighted.}
		\label{fig:symmetric_4_clo}
	\end{subfigure}
	
	\vspace*{.5cm}
	
	\begin{subfigure}[t]{.45\textwidth}
		\centering
		\includegraphics[scale=.7,page=18]{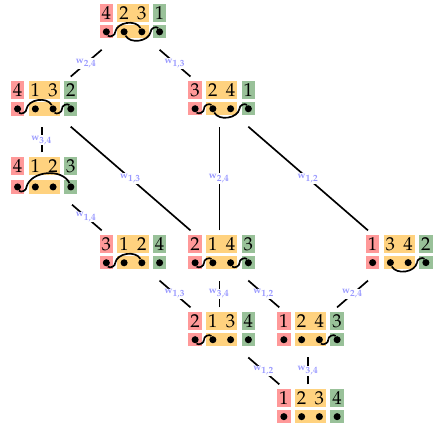}
		\caption{The lattice $\Tamari\bigl((1,2,1)\bigr)$ labeled by \eqref{eq:cu_labeling}.}
		\label{fig:tamari_121_labeled}
	\end{subfigure}
	\hspace*{1cm}
	\begin{subfigure}[t]{.45\textwidth}
		\centering
		\includegraphics[scale=.7,page=19]{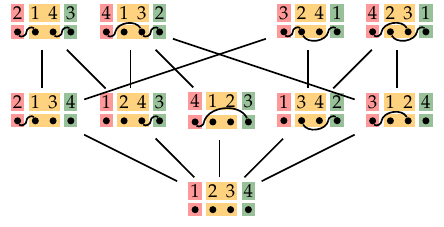}
		\caption{The core label order of $\Tamari\bigl((1,2,1)\bigr)$.}
		\label{fig:tamari_121_clo}
	\end{subfigure}
	\caption{Two lattices of $(1,2,1)$-permutations and their core label orders.}
	\label{fig:alpha_clo}
\end{figure}

Since both $\Weak(\Symmetric_{\alpha})$ and $\Tamari(\alpha)$ are congruence-uniform lattices, it makes sense to distinguish the corresponding core label sets.  For $u\in\Symmetric_{\alpha}(231)$, we write $\Psi_{L}(u)$ for the core label set in $\Weak(\Symmetric_{\alpha})$, and $\Psi_{\alpha}(u)$ for the core label set in $\Tamari(\alpha)$.

We now show that for $u\in\Symmetric_{\alpha}(231)$, the core label set $\Psi_{\alpha}(u)$ induces a noncrossing $\alpha$-partition.  To that end, we define
\begin{displaymath}
	X(u) \defs \bigl\{w_{a,b}\mid a\sim_{\BijNCAlign_{\alpha}(u)}b\bigr\}.
\end{displaymath}

\begin{proposition}\label{prop:core_label_set_is_noncrossing}
	Let $\alpha$ be a composition of $n>0$.  For all $u\in\Symmetric_{\alpha}(231)$, $\Psi_{\alpha}(u)\subseteq X(u)$.
\end{proposition}
\begin{proof}
	By Theorem~\ref{thm:alpha_tamari_lattice}, $\Tamari(\alpha)$ is a quotient lattice of $\Weak(\Symmetric_{\alpha})$  by a lattice congruence $\Theta_{\alpha}$.  
	
	Let $u\in\Symmetric_{\alpha}(231)$.  If $\Psi_{\alpha}(u)=\emptyset$, then there is nothing to show.  Otherwise, Theorem~\ref{thm:join_representation_lower_covers} implies that $\Des(u)=\bigl\{(a_{1},b_{1}),(a_{2},b_{2}),\ldots,(a_{t},b_{t})\bigr\}\neq\emptyset$.  We denote by $G$ the subgroup of $\Symmetric_{\alpha}$ generated by the transpositions corresponding to these descents.
	
	Since $\Psi_{\alpha}(u)\neq\emptyset$, we may pick any $w_{a,b}\in\Psi_{\alpha}(u)$.  By Lemma~\ref{lem:core_label_sets_quotients}, $w_{a,b}\in\Psi_{L}(u)$.  Since $\Weak(\Symmetric_{\alpha})$ is a principal order ideal in $\Weak(\Symmetric_{n})$ and $\Symmetric_{n}$ is a Coxeter group, \cite{stump18cataland}*{Theorem~2.10.5} thus implies that $w_{a,b}\in G$ and $\Inv(w_{a,b})\subseteq\Inv(u)$.  Since $w_{a,b}\in G$, we may write the transposition swapping $a$ and $b$ as a product of the generators of $G$.  This implies $a\sim_{\BijNCAlign_{\alpha}(u)}b$, and therefore $w_{a,b}\in X(u)$.
\end{proof}

\begin{example}\label{ex:strict_inclusion}
	Let $\alpha=(1,2,1)$ and consider $u=\oneline{3\mid 2\;4\mid 1}\in\Symmetric_{\alpha}$.  Then, $\BijNCAlign(u)=\bigl\{\{1,2,4\},\{3\}\bigr\}$ and therefore $X(u)=\bigl\{w_{1,2},w_{1,4},w_{2,4}\bigr\}$.  The subgroup $G$ from the proof of Proposition~\ref{prop:core_label_set_is_noncrossing} is generated by $w_{1,2}$ and $w_{2,4}$.  It follows that $X(u)\subseteq G$.
	
	We immediately see that $\Inv(u)=\bigl\{(1,2),(1,4),(2,4),(3,4)\bigr\}$.  Moreover, we obtain from Corollary~\ref{cor:alpha_irreducibles} that
	\begin{align*}
		\Inv(w_{1,2}) & = \bigl\{(1,2)\bigr\},\\
		\Inv(w_{1,4}) & = \bigl\{(1,2),(1,3),(1,4)\bigr\},\\
		\Inv(w_{2,4}) & = \bigl\{(2,4),(3,4)\bigr\}.
	\end{align*}
	Thus, $\Inv(w_{1,2})\subseteq\Inv(u)$ and $\Inv(w_{2,4})\subseteq\Inv(u)$, but $\Inv(w_{1,4})\not\subseteq\Inv(u)$.  Now, since $w_{a,b}\in\Psi_{L}(u)$ if and only if $w_{a,b}\in G$ and $\Inv(w_{a,b})\subseteq\Inv(u)$ we conclude that $w_{1,4}\notin\Psi_{L}(u)$.  By Lemma~\ref{lem:core_label_sets_quotients}, $w_{1,4}\notin\Psi_{\alpha}(u)$.  
	
	By inspection of Figure~\ref{fig:symmetric_4_labeled}, we observe that $\Psi_{L}(u)$ contains the irreducible permutations $j_{1}=\oneline{3\mid 1\;4\mid 2}$ and $j_{2}=\oneline{2\mid 3\;4\mid 1}$, both of which contain an $(\alpha,231)$-pattern in positions $(1,3,4)$.
\end{example}

The next proposition characterizes the compositions for which equality holds in Proposition~\ref{prop:core_label_set_is_noncrossing}.

\begin{proposition}\label{prop:core_label_set_is_noncrossing_partition}
	Let $\alpha$ be a composition of $n>0$.  Then, $\Psi_{\alpha}(u)=X(u)$ for all $u\in\Symmetric_{\alpha}(231)$ if and only if either $\alpha=(n)$ or $\alpha=(p,1,1,\ldots,1,q)$ for some integers $p,q>0$.
\end{proposition}
\begin{proof}
	If $\alpha=(n)$, then $\Symmetric_{\alpha}(231)=\{\id\}$ and $\Psi_{\alpha}(\id)=\emptyset=X(\id)$.  Now, suppose that $\alpha=(p,1,1,\ldots,1,q)$ for some integers $p,q>0$.  Let $u\in\Symmetric_{\alpha}(231)$ with $\Des(u)=\bigl\{(a_{1},b_{1}),(a_{2},b_{2}),\ldots,(a_{t},b_{t})\bigr\}$.
	
	By Proposition~\ref{prop:core_label_set_is_noncrossing}, $\Psi_{\alpha}(u)\subseteq X(u)$.  In order to show the reverse inclusion, we pick 
	$w_{a,b}\in X(u)$ and prove that $w_{a,b}\in\Psi_{\alpha}(u)$.  Using \cite{stump18cataland}*{Theorem~2.10.5} as in the proof of Proposition~\ref{prop:core_label_set_is_noncrossing} it is enough to show that $\Inv(w_{a,b})\subseteq\Inv(u)$.
	
	By definition, there exists a sequence of integers $k_{0},k_{1},\ldots,k_{t}$ such that $a=k_{0}$ and $b=k_{t}$ and $(k_{i-1},k_{i})$ is a bump of $\BijNCAlign_{\alpha}(u)$ for all $i\in[t]$.  In particular, all the $k_{i}$ lie in different $\alpha$-regions.  By Theorem~\ref{thm:nc_231_bijection}, $(k_{i-1},k_{i})\in\Des(u)$ for all $i\in[t]$.  Thus, $(a,b)\in\Inv(u)$.
	
	If $t=1$, then $(a,b)\in\Des(u)$.  By Proposition~\ref{prop:alpha_join_representations}, $w_{a,b}$ is a canonical joinand of $u$, which implies $w_{a,b}\in\Psi_{\alpha}(u)$.
	
	If $t>1$, then we consider two cases.  If $a>p$, then 
	\begin{displaymath}
		\Inv(w_{a,b}) = \bigl\{(a,a{+}1),(a,a{+}2),\ldots,(a,b)\bigr\}
	\end{displaymath}
	by Corollary~\ref{cor:alpha_irreducibles} and our assumption on the shape of $\alpha$.  Let $d\in\{a{+}1, a{+}2, \ldots, b\}$.  By construction, there exists $(k_{i-1},k_{i})\in\Des(u)$ such that $k_{i-1}<d\leq k_{i}$.  Since $u$ avoids any $(\alpha,231)$-pattern, it follows that $u_{d}<u_{k_{i-1}}<u_{k_{i-2}}<\cdots<u_{k_{0}}=u_{a}$.  Thus, $(a,d)\in\Inv(u)$.  It follows that $\Inv(w_{a,b})\subseteq\Inv(u)$ as desired.
	
	If $p\leq a$, then Corollary~\ref{cor:alpha_irreducibles} implies
	\begin{displaymath}
		\Inv(w_{a,b}) = \bigl\{(a',b')\mid a'\in\{a,a{+}1,\ldots,p\}, b'\in\{p{+}1,p{+}2,\ldots,b\}\bigr\}.
	\end{displaymath}
	As before we may show that $(a,d)\in\Inv(u)$ for any $d\in\{p{+}1,p{+}2,\ldots,b\}$.  Since $u\in\Symmetric_{\alpha}$, we have $u_{a}<u_{a'}$ for any $a'\in\{a{+}1,a{+}2,\ldots,p\}$.  This implies $\Inv(w_{a,b})\subseteq\Inv(u)$.
	
	We conclude that $w_{a,b}\in\Psi_{L}(u)$.  Since, by construction, $w_{a,b}\in\Symmetric_{\alpha}(231)$, it follows that $w_{a,b}\in\Psi_{\alpha}(u)$.
	
	\medskip
	
	Now suppose that $\alpha=(\alpha_{1},\alpha_{2},\ldots,\alpha_{r})$ is a composition of $n$ which is not of the form $\alpha=(p,1,1,\ldots,q)$.  Then, $r\geq 3$ and there exists $k\in\{2,3,\ldots,r-1\}$ such that $p_{k}>p_{k-1}+1$.  

	Let $a=p_{k-1}$ and consider $\Pf\in\NC(\alpha)$ whose only non-singleton blocks are $\{a,a{+}1\}$ and $\{a{+}1,n\}$, and let $u=\BijNCAlign_{\alpha}^{-1}(\Pf)$.  Then, $u$ has one-line notation
	\begin{displaymath}
		\underbrace{1,2,\ldots,a{-}1,n{-}\alpha_{k}{+}1}_{p_{k-1}}\mid \underbrace{n{-}\alpha_{k},n{-}\alpha_{k}{+}2,\ldots,n}_{\alpha_{k}}\mid \underbrace{a,a{+}1,\ldots,n{-}\alpha_{k}{-}1}_{n-p_{k}}.
	\end{displaymath}
	
	By construction, the join-irreducible permutation $w_{a,n}\in\Symmetric_{\alpha}(231)$ is contained in $X(u)$.  By Corollary~\ref{cor:alpha_irreducibles}, 
	\begin{displaymath}
		\Inv(w_{a,n}) = \bigl\{(a,a{+}1),(a,a{+}2),\ldots,(a,n)\bigr\};
	\end{displaymath}
	in particular $(a,p_{k})\in\Inv(w_{a,n})$.  However, we notice in the one-line notation of $u$ that $(a,p_{k})\notin\Inv(u)$, because $\alpha_{k}=p_{k}-p_{k-1}>1$.  It follows that $\Inv(w_{a,n})\not\subseteq\Inv(u)$ and therefore $w_{a,n}\notin\Psi_{\alpha}(u)$.
\end{proof}

We now relate the core label order of $\Tamari(\alpha)$ to the refinement order on $\NC(\alpha)$.  Given two partitions $\Pf_{1},\Pf_{2}\in\Pi_{\alpha}$, we say that $\Pf_{1}$ \defn{refines} $\Pf_{2}$ if every block of $\Pf_{1}$ is contained in some block of $\Pf_{2}$; we write $\Pf_{1}\refines\Pf_{2}$ in that case.

\begin{lemma}\label{lem:refinement_by_irreducibles}
	For $u,v\in\Symmetric_{\alpha}(231)$, $\BijNCAlign_{\alpha}(u)\refines\BijNCAlign_{\alpha}(v)$ if and only if $X(u)\subseteq X(v)$.
\end{lemma}
\begin{proof}
	Suppose that $\BijNCAlign_{\alpha}(u)\refines\BijNCAlign_{\alpha}(v)$ and pick $w_{a,b}\in X(u)$.  By definition, $a\sim_{\BijNCAlign_{\alpha}(u)}b$, and thus $a\sim_{\BijNCAlign_{\alpha}(v)}b$.  Hence, $w_{a,b}\in X(v)$.
	
	Conversely, suppose that $X(u)\subseteq X(v)$ and pick $a,b\in[n]$ with $a\sim_{\BijNCAlign_{\alpha}(u)}b$.  By definition, $w_{a,b}\in X(u)\subseteq X(v)$, and thus $a\sim_{\BijNCAlign_{\alpha}(v)}b$.  This implies $\BijNCAlign(u)\refines\BijNCAlign_{\alpha}(v)$.
\end{proof}

\begin{theorem}\label{thm:alpha_tamari_clo_nc_poset}
	Let $\alpha$ be a composition of $n$.  The core label order of $\Tamari(\alpha)$ is isomorphic to $\bigl(\NC(\alpha),\refines\bigr)$ if and only if $\alpha=(n)$ or $\alpha=(p,1,1,\ldots,q)$ for some integers $p,q>0$.
\end{theorem}
\begin{proof}
	Let $u,v\in\Symmetric_{\alpha}(231)$.  By Lemma~\ref{lem:refinement_by_irreducibles}, $\BijNCAlign_{\alpha}(u)\refines\BijNCAlign_{\alpha}(v)$ if and only if $X(u)\subseteq X(v)$.  By definition of the core label order (see Section~\ref{sec:core_label_order}), $u\sqsubseteq v$ if and only if $\Psi_{\alpha}(u)\subseteq\Psi_{\alpha}(v)$.  Now, Proposition~\ref{prop:core_label_set_is_noncrossing_partition} states that $X(u)=\Psi_{\alpha}(u)$ and $X(v)=\Psi_{\alpha}(v)$ if and only if $\alpha=(n)$ or $\alpha=(p,1,1,\ldots,1,q)$ for some integers $p,q>0$.
\end{proof}

Figure~\ref{fig:tamari_212} shows $\Tamari\bigl((2,1,2)\bigr)$, and Figure~\ref{fig:tamari_212_clo} shows $\CLO\bigl(\Tamari\bigl((2,1,2)\bigr)\bigr)$.  This illustrates Theorem~\ref{thm:alpha_tamari_clo_nc_poset}, since we can verify directly that $\CLO\bigl(\Tamari\bigl((2,1,2)\bigr)\bigr))$ is indeed isomorphic to $\bigl(\NC\bigl((2,1,2)\bigr),\refines\bigr)$.  

In contrast, Figure~\ref{fig:tamari_121_clo} shows $\CLO\bigl(\Tamari\bigl((1,2,1)\bigr)\bigr)$ and this poset is \emph{not} isomorphic to $\bigl(\NC\bigl((1,2,1)\bigr),\refines\bigr)$.  If $u=\oneline{4\mid 1\;2\mid 3}$ and $v=\oneline{3\mid 2\;4\mid 1}$, then $\BijNCAlign_{(1,2,1)}(u)\refines\BijNCAlign_{(1,2,1)}(v)$, but $u\not\sqsubseteq v$; see also Example~\ref{ex:strict_inclusion}.

\begin{figure}
	\centering
	\includegraphics[scale=1,page=21]{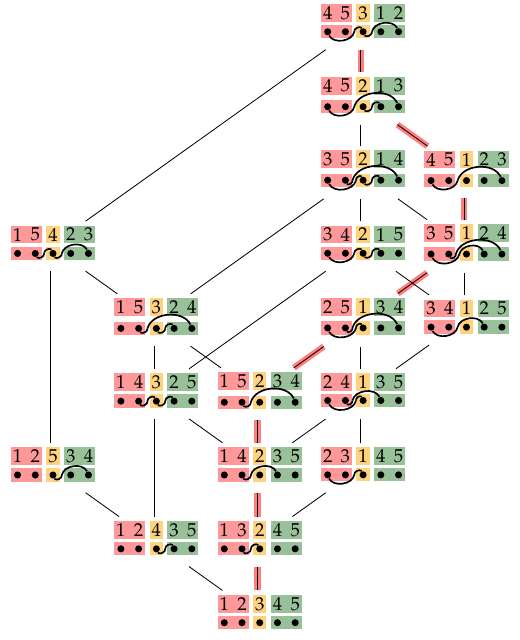}
	\caption{The lattice $\Tamari\bigl((2,1,2)\bigr)$.}
	\label{fig:tamari_212}
\end{figure}

\begin{figure}
	\centering
	\includegraphics[scale=.8,page=22]{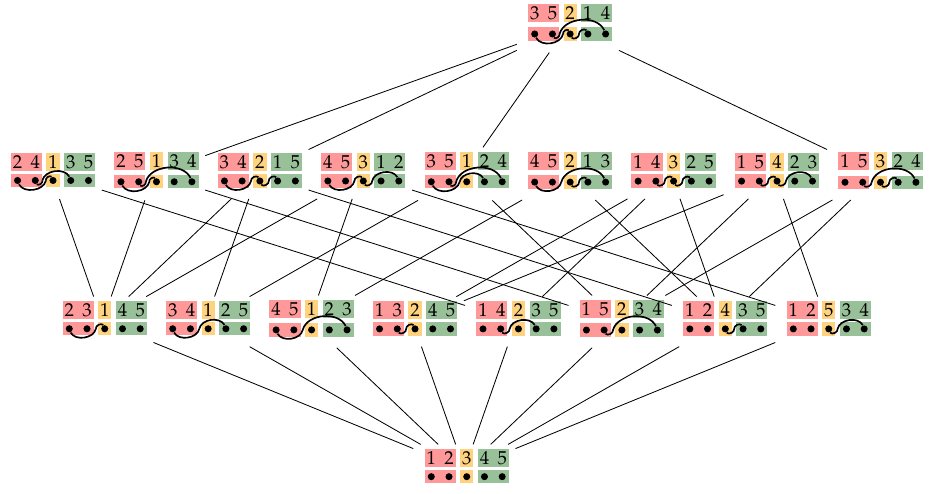}
	\caption{The core label order of $\Tamari\bigl((2,1,2)\bigr)$.  This is also the poset $\bigl(\NC\bigl((2,1,2)\bigr),\refines\bigr)$.}
	\label{fig:tamari_212_clo}
\end{figure}

We conclude this section with the observation that the core label order of $\Tamari(\alpha)$ is always a meet-semilattice.  Recall the definition of the intersection property from Section~\ref{sec:core_label_order}.

\begin{theorem}\label{thm:alpha_tamari_intersection_property}
	For all $n>0$ and every composition $\alpha$ of $n$, $\Tamari(\alpha)$ has the intersection property.
\end{theorem}
\begin{proof}
	Let $w\in\Symmetric_{\alpha}$.  If we denote the core label set of $w$ in $\Weak(\Symmetric_{n})$ by $\Psi_{L;n}$ and the core label set of $w$ in $\Weak(\Symmetric_{\alpha})$ by $\Psi_{L;\alpha}$, then $\Psi_{L;n}(w)=\Psi_{L;\alpha}(w)$, because $\Weak(\Symmetric_{\alpha})$ is principal order ideal of $\Weak(\Symmetric_{n})$ by Theorem~\ref{thm:alpha_weak_order_lattice}.  
	
	For $j\in\JI\bigl(\Weak(\Symmetric_{n})\bigr)$, if $j\in\Psi_{L;\alpha}(w)$, then $j\leq_{L}w$ by Corollary~\ref{cor:minimal_cover_irreducible}.  This means that $j\in\JI\bigl(\Weak(\Symmetric_{\alpha})\bigr)$.  Thus, $\CLO\bigl(\Weak(\Symmetric_{\alpha})\bigr)$ is an order ideal of $\CLO\bigl(\Weak(\Symmetric_{n})\bigr)$.
	
	By \cite{reading11noncrossing}*{Proposition~5.1} (see also \cite{bancroft11shard}*{Section~4}), $\CLO\bigl(\Weak(\Symmetric_{n})\bigr)$ is a lattice, which means that $\CLO\bigl(\Weak(\Symmetric_{\alpha})\bigr)$ is a meet-semilattice.  Thus, by Theorem~\ref{thm:intersection_property}, $\Weak(\Symmetric_{\alpha})$ has the intersection property.  Now, Proposition~\ref{prop:intersection_property_quotients} implies that any quotient lattice of $\Weak(\Symmetric_{\alpha})$ has the intersection property.  By Theorem~\ref{thm:alpha_tamari_lattice}, this is the case for $\Tamari(\alpha)$.
\end{proof}

In a preliminary draft of this article, we claimed that the poset ${\bigl(\NC(\alpha),\refines\bigr)}$ is a meet-semilattice.  A referee has provided the following counterexample.  

\begin{example}\label{ex:alpha_refinement_no_semilattice}
	Let $\alpha=(2,4,3,1)$, and consider 
	\begin{align*}
		\Pf_{1} & = \bigl\{\{1,3,8,10\},\{2,6,9\},\{4\},\{5\},\{7\}\bigr\},\\
		\Pf_{2} & = \bigl\{\{1,4,7,10\},\{2,5,9\},\{3\},\{6\},\{8\}\bigr\}.
	\end{align*}
	Then, $\Pf_{1},\Pf_{2}\in\NC(\alpha)$, but their intersection is 
	\begin{displaymath}
		\Pf = \bigl\{\{1,10\},\{2,9\},\{3\},\{4\},\{5\},\{6\},\{7\},\{8\}\bigr\}\notin\NC(\alpha).
	\end{displaymath}
	Let $\mathbf{Q}_{1}=\BijNCAlign_{\alpha}(w_{1,10})$ and $\mathbf{Q}_{2}=\BijNCAlign_{\alpha}(w_{2,9})$.  Then $\mathbf{Q}_{1},\mathbf{Q}_{2}\in\NC(\alpha)$ and $\mathbf{Q}_{i}\refines\Pf_{j}$ for $i,j\in\{1,2\}$.  Thus, ${\bigl(\NC\bigl((2,4,3,1)\bigr),\refines\bigr)}$ is \emph{not} a meet-semilattice.
\end{example}

At the moment, we do not have anything meaningful to say about the posets ${\bigl(\NC(\alpha),\refines\bigr)}$, except for the cases in which they coincide with $\CLO\bigl(\Tamari(\alpha)\bigr)$.

\section{The $\alpha$-Tamari lattices are trim}
	\label{sec:alpha_tamari_trim}
In this section, we prove that $\Tamari(\alpha)$ is trim for every composition $\alpha$ of $n>0$.

\begin{proposition}\label{prop:alpha_tamari_trim}
	For all $n>0$ and every composition $\alpha$ of $n$, the lattice $\Tamari(\alpha)$ is trim.
\end{proposition}

We first study the join-irreducible elements of $\Tamari(\alpha)$ in greater detail.

\subsection{The poset of irreducibles of $\Tamari(\alpha)$}
	\label{sec:alpha_tamari_irreducibles}
Recall from Corollary~\ref{cor:arcs_join_irreducibles} that the join-irreducible elements of $\Tamari(\alpha)$ are in bijection with the $\alpha$-arcs.  Moreover, Corollary~\ref{cor:alpha_irreducibles} describes one-line notation and inversion sets of the join-irreducibles, and immediately implies the next result.

\begin{corollary}\label{cor:alpha_irreducibles_comparison}
	Let $w_{a,b},w_{a',b'}\in\JI\bigl(\Tamari(\alpha)\bigr)$.  Then, $w_{a,b}\leq_{L}w_{a',b'}$ if and only if $a$ and $a'$ belong to the same $\alpha$-region and $a'\leq a<b\leq b'$.
\end{corollary}

We may now describe the restriction of the weak order to the set $\JI\bigl(\Tamari(\alpha)\bigr)$.  See Figure~\ref{fig:alpha_irreducibles} for an illustration.

\begin{figure}
	\centering
	\begin{subfigure}[t]{.45\textwidth}
		\centering
		\includegraphics[scale=1,page=29]{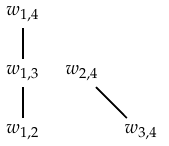}
		\caption{The set $\JI\bigl(\Tamari\bigl((1,2,1)\bigr)\bigr)$ ordered by weak order.}
		\label{fig:alpha_irreducibles_121}
	\end{subfigure}
	\hspace*{1cm}
	\begin{subfigure}[t]{.45\textwidth}
		\centering
		\includegraphics[scale=1,page=24]{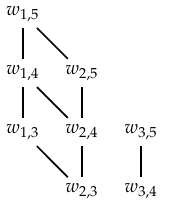}
		\caption{The set $\JI\bigl(\Tamari\bigl((2,1,2)\bigr)\bigr)$ ordered by weak order.}
		\label{fig:alpha_irreducibles_212}
	\end{subfigure}
	\caption{Two posets of join-irreducible permutations.}
	\label{fig:alpha_irreducibles}
\end{figure}

\begin{proof}[Proof of Theorem~\ref{thm:alpha_tamari_irreducibles_poset}]
	By Corollary~\ref{cor:alpha_irreducibles_comparison} we conclude that for $w_{a,b}\leq_{L} w_{a',b'}$ to hold, it is necessary that $a$ and $a'$ belong to the same $\alpha$-region.  This accounts for the $r-1$ connected components of $\Weak\bigl(\JI\bigl(\Tamari(\alpha)\bigr)\bigr)$, because $a$ can be chosen from any but the last $\alpha$-region and there is a total of $r$ $\alpha$-regions.
	
	Now suppose that $a$ lies in the $j\th$ $\alpha$-region, which means that $a$ takes any of the values $\{p_{j-1}{+}1,p_{j-1}{+}2,\ldots,p_{j}\}$.  For any choice of $a$, we can pick some $b\in\{p_{j}{+}1,p_{j}{+}2,\ldots,n\}$ to obtain a join-irreducible element $w_{a,b}$.  Observe that whenever $a\neq p_{j-1}{+}1$, then $w_{a,b}\leq_{L}w_{a-1,b}$, and we always have $w_{a,b}\leq_{L} w_{a,b+1}$ when $b<n$.  This implies that the $j\th$ component of $\Weak\bigl(\JI\bigl(\Tamari(\alpha)\bigr)\bigr)$ is isomorphic to the direct product of an $\alpha_{j}$-chain and an $(\alpha_{j+1}{+}\alpha_{j+2}{+}\cdots{+}\alpha_{r})$-chain.
\end{proof}

If $\alpha=(1,1,\ldots,1)$, then Theorem~\ref{thm:alpha_tamari_irreducibles_poset} states that the poset of irreducibles of the ordinary Tamari lattice is a union of $n-1$ chains of lengths $1,2,\ldots,n-1$, respectively.  This result was previously found in \cite{bennett94two}*{Theorem~11}.

\begin{corollary}\label{cor:alpha_irreducibles_number}
	Let $\alpha=(\alpha_{1},\alpha_{2},\ldots,\alpha_{r})$ be a composition of $n>0$.  Then,
	\begin{displaymath}
		\Bigl\lvert\JI\bigl(\Tamari(\alpha)\bigr)\Bigr\rvert = \sum_{j=1}^{r-1}\alpha_{j}\cdot\bigl(\alpha_{j+1}{+}\alpha_{j+2}{+}\cdots{+}\alpha_{r}\bigr).
	\end{displaymath}
\end{corollary}

We now show that $\Tamari(\alpha)$ is trim for every composition $\alpha$.  See Section~\ref{sec:trim_lattices} for the necessary definitions.

\begin{proposition}\label{prop:alpha_tamari_extremal}
	For all $n>0$ and every composition $\alpha$ of $n$, the lattice $\Tamari(\alpha)$ is extremal.
\end{proposition}
\begin{proof}
	Let
	\begin{equation}\label{eq:alpha_tamari_length}
		f(\alpha) \defs \sum_{j=1}^{r-1}\alpha_{j}\cdot\bigl(\alpha_{j+1}{+}\alpha_{j+2}{+}\cdots{+}\alpha_{r}\bigr).
	\end{equation}
	By Corollary~\ref{cor:alpha_tamari_semidistributive}, $\Tamari(\alpha)$ is semidistributive, which implies $\bigl\lvert\JI\bigl(\Tamari(\alpha)\bigr)\bigr\rvert=\bigl\lvert\MI\bigl(\Tamari(\alpha)\bigr)\bigr\rvert$ by Lemma~\ref{lem:kappa_bijection}.  By Corollary~\ref{cor:alpha_irreducibles_number}, $\bigl\lvert\JI\bigl(\Tamari(\alpha)\bigr)\bigr\rvert=f(\alpha)$.  By \eqref{eq:lattice_irreducibles_length}, it remains to exhibit a chain in $\Tamari(\alpha)$ consisting of $f(\alpha)+1$ elements.
	
	Let $\alpha=(\alpha_{1},\alpha_{2},\ldots,\alpha_{r})$.  We apply induction on $r$.  If $r=1$, then $\alpha=(n)$ and $\Symmetric_{\alpha}(231)=\{\id\}$.  Thus $\Tamari(\alpha)$ is the singleton lattice which is trivially trim.
	
	Now assume that the claim is true for all compositions of $n$ with at most $r-1$ parts, and recall that $p_{j}=\alpha_{1}+\alpha_{2}+\cdots+\alpha_{j}$ for $j\in[r]$.
	
	We set $v^{(0,0)}=\id$ and for $k\in[n{-}p_{1}]$ we define $v^{(0,k)}=s_{p_{1}+k-1}\circ v^{(0,k-1)}$.  This means that if the value of $v^{(0,k-1)}$ in position $p_{1}$ is $a$, then we move to $v^{(0,k)}$ by swapping the values $a$ and $a+1$.  Since we do this in order from left to right, $v^{(0,k-1)}\lessdot_{L}v^{(0,k)}$ and $v^{(0,k)}\in\Symmetric_{\alpha}(231)$ for all $k\in[n{-}p_{1}]$.  Then, $v^{(0,n-p_{1})}$ has the one-line notation
	\begin{displaymath}
		\oneline{1,2,\ldots,p_{1}{-}1,n\mid p_{1},p_{1}{+}2,\ldots,n{-}1}.
	\end{displaymath}
	(The vertical bar indicates the end of the first $\alpha$-region.)
	
	Now, for $i\in[p_{1}{-}1]$, we set $v^{(i,1)}=s_{p_{1}-i}\circ v^{(i-1,n-p_{1})}$ (which means that we swap the values $p_{1}{-}i$ and $p_{1}{-}i{+}1$), and for $k\in\{2,3,\ldots,n{-}p_{1}\}$ we set $v^{(i,k)}=s_{p_{1}-i+k-1}\circ v^{(i,k-1)}$.  As before, each of these elements is $(\alpha,231)$-avoiding.  Then, $v^{(p_{1}-1,n-p_{1})}$ has the one-line notation
	\begin{displaymath}
		\oneline{n{-}p_{1}{+}1,n{-}p_{1}{+}2,\ldots,n\mid 1,2,\ldots,n-p_{1}}.
	\end{displaymath}
	
	This constitutes a chain of length $p_{1}\cdot(n-p_{1})$ from $\id$ to $v^{(p_{1}-1,n-p_{1})}$ in $\Tamari(\alpha)$.  The interval $\bigl[v^{(p_{1}-1,n-p_{1})},\woa\bigr]$ in $\Tamari(\alpha)$ is isomorphic to $\Tamari\bigl((\alpha_{2},\ldots,\alpha_{r})\bigr)$, which by induction has length $f\bigl((\alpha_{2},\ldots,\alpha_{r})\bigr)$.  It follows that 
	\begin{align*}
		\ell\bigl(\Tamari(\alpha)\bigr) & = p_{1}\cdot(n-p_{1})+\sum_{j=2}^{r-1}\alpha_{j}\cdot(\alpha_{j+1}+\alpha_{j+2}+\cdots+\alpha_{r})\\
		& = \alpha_{1}\cdot(\alpha_{2}+\alpha_{3}+\cdots+\alpha_{r})+\sum_{j=2}^{r-1}\alpha_{j}\cdot(\alpha_{j+1}+\alpha_{j+2}+\cdots+\alpha_{r})\\
		& = \sum_{j=1}^{r-1}\alpha_{j}\cdot(\alpha_{j+1}+\alpha_{j+2}+\cdots+\alpha_{r})\\
		& = f(\alpha).
	\end{align*}
	Hence, $\Tamari(\alpha)$ is extremal.
\end{proof}

For $\alpha=(2,1,2)$, the maximal chain constructed in the proof of Proposition~\ref{prop:alpha_tamari_extremal} is highlighted in Figure~\ref{fig:tamari_212}.  We may now conclude the proof of Proposition~\ref{prop:alpha_tamari_trim}.

\begin{proof}[Proof of Proposition~\ref{prop:alpha_tamari_trim}]
	By Corollary~\ref{cor:alpha_tamari_semidistributive}, $\Tamari(\alpha)$ is semidistributive, and by Proposition~\ref{prop:alpha_tamari_extremal}, $\Tamari(\alpha)$ is extremal.  Then, Theorem~\ref{thm:semidistributive_extremal_trim} implies that $\Tamari(\alpha)$ is trim.
\end{proof}

\begin{corollary}\label{cor:alpha_irreducibles_order}
	Let $C$ be the maximal chain constructed in the proof of Proposition~\ref{prop:alpha_tamari_extremal}, and let $\lambda$ be the labeling from \eqref{eq:cu_labeling}.  The labels appearing on $C$ are pairwise distinct and they induce a total order on $\JI\bigl(\Tamari(\alpha)\bigr)$ given by the following cover relations:
	\begin{displaymath}
		w_{a,b} \prec \begin{cases}
			w_{a,b+1}, & \text{if}\; p_{j}+1\leq b<n,\\ 
			w_{a-1,p_{j}+1}, & \text{if}\;a\neq p_{j-1}+1\;\text{and}\;b=n,\\
			w_{p_{j+1},p_{j+1}+1}, & \text{if}\;a=p_{j-1}+1\;\text{and}\;b=n,\\
		\end{cases}
	\end{displaymath}
	if $a$ belongs to the $j\th$ $\alpha$-region.
\end{corollary}
\begin{proof}
	With the notation from the proof of Proposition~\ref{prop:alpha_tamari_extremal}, the first $p_{1}\cdot (n-p_{1})$ cover relations along $C$ are:
	\begin{multline*}
		\id = v^{(0,0)} \lessdot_{\alpha} v^{(0,1)} \lessdot_{\alpha} \cdots \lessdot_{\alpha} v^{(0,n-p_{1})} \lessdot_{\alpha} v^{(1,1)} \lessdot_{\alpha} \cdots\\
			\lessdot_{\alpha} v^{(1,n-p_{1})} \lessdot_{\alpha} \cdots \lessdot_{\alpha} v^{(2,1)} \lessdot_{\alpha} \cdots \lessdot_{\alpha} v^{(p_{1}-1,n-p_{1})}.
	\end{multline*}
	By construction, $C$ is also a maximal chain in $\Weak(\Symmetric_{\alpha})$ and it follows that 
	\begin{displaymath}
		\lambda\Bigl(v^{(i,k)},v^{(i,k+1)}\Bigr) = \begin{cases}
			w_{p_{1},p_{1}+1}, & \text{if}\;i=k=0,\\ 
			w_{p_{1}-i-1,p_{1}+1}, & \text{if}\;0\leq i<p_{1}-1,k=n-p_{1},\\
			w_{p_{1}-i,p_{1}+k} & \text{if}\;0\leq i\leq p_{1}-1,0<k<n-p_{1}.
		\end{cases}
	\end{displaymath}
	(If $k=n-p_{1}$, then we set $k+1=1$.)  The claim follows by induction.
\end{proof}

\begin{example}
	Let $\alpha=(2,1,2)$.  The chain constructed in the proof of Proposition~\ref{prop:alpha_tamari_extremal} is highlighted in Figure~\ref{fig:tamari_212}.  The total order of the join-irreducibles of $\Tamari\bigl((2,1,2)\bigr)$ is
	\begin{displaymath}
		w_{2,3}\;\prec\; w_{2,4}\;\prec\; w_{2,5}\;\prec\; w_{1,3}\;\prec\; w_{1,4}\;\prec\; w_{1,5}\;\prec\; w_{3,4}\;\prec\; w_{3,5}.
	\end{displaymath}
\end{example}

\begin{remark}
	If $\alpha=(1,1,\ldots,1)$ is a composition of $n$, then the join-irreducibles of $\Tamari\bigl((1,1,\ldots,1)\bigr)=\Tamari(n)$ correspond to all transpositions $(a,b)$ for $1\leq a<b\leq n$.  The total order defined in Corollary~\ref{cor:alpha_irreducibles_order} corresponds to the lexicographic order on these transpositions.
	
	This order corresponds to the so-called \emph{inversion order} of the longest element $\wo\in\Symmetric_{n}$ with respect to the linear Coxeter element.  It seems that this correspondence works in general, \ie the order defined in Corollary~\ref{cor:alpha_irreducibles_order} recovers the inversion order of the parabolic longest element $\woa\in\Symmetric_{\alpha}$ with respect to the linear Coxeter element.
\end{remark}

We conclude this section with the proof of Theorem~\ref{thm:alpha_tamari_structure}.

\begin{proof}[Proof of Theorem~\ref{thm:alpha_tamari_structure}]
	$\Tamari(\alpha)$ is congruence uniform by Proposition~\ref{prop:alpha_tamari_congruence_uniform} and trim by Proposition~\ref{prop:alpha_tamari_trim}.
\end{proof}

\subsection{The Galois graph of $\Tamari(\alpha)$}
	\label{sec:alpha_tamari_galois}
By Proposition~\ref{prop:alpha_tamari_extremal}, $\Tamari(\alpha)$ is an extremal lattice.  Any extremal lattice can be described in terms of a directed graph; its Galois graph, see Section~\ref{sec:galois_graph}.  

In this section, we give an explicit description of the Galois graph of $\Tamari(\alpha)$.  We exploit the fact from Proposition~\ref{prop:alpha_tamari_congruence_uniform} that $\Tamari(\alpha)$ is also congruence uniform.  

Let us recall the following useful characterization of inversion sets of joins in the weak order.

\begin{lemma}[\cite{markowsky94permutation}*{Theorem~1(b)}]\label{lem:inversion_set_join}
	Let $u,v\in\Symmetric_{n}$.  The inversion set $\Inv(u\vee_{L}v)$ is the transitive closure of $\Inv(u)\cup\Inv(v)$, \ie if $(a,b),(b,c)\in\Inv(u)\cup\Inv(v)$, then $(a,c)\in\Inv(u\vee v)$. 
\end{lemma}

\begin{proof}[Proof of Theorem~\ref{thm:alpha_tamari_galois_graph}]
	By definition, the vertex set of $\Galois\bigl(\Tamari(\alpha)\bigr)$ is $[K]$, where
	\begin{displaymath}
		K = \bigl\lvert\JI\bigl(\Tamari(\alpha)\bigr)\bigr\rvert = f(\alpha)
	\end{displaymath}
	and $f(\alpha)$ is defined in \eqref{eq:alpha_tamari_length}.  There exists a directed edge $s\to t$ in $\Galois\bigl(\Tamari(\alpha)\bigr)$ if $s\neq t$ and $j_{s}\not\leq m_{t}$, where the join- and meet-irreducible elements of $\Tamari(\alpha)$ are ordered as in \eqref{eq:extremal_order}.  By Proposition~\ref{prop:alpha_tamari_congruence_uniform}, $\Tamari(\alpha)$ is also congruence uniform, so that Corollary~\ref{cor:kappa_extremal_order}(ii) implies $s\to t$ if and only if $s\neq t$ and $j_{t}\leq {j_{t}}_{*}\vee j_{s}$.  We may thus view $\Galois\bigl(\Tamari(\alpha)\bigr)$ as a directed graph on the vertex set $\JI\bigl(\Tamari(\alpha)\bigr)$.  
	
	Now, pick $w_{a,b},w_{a',b'}\in\JI\bigl(\Tamari(\alpha)\bigr)$ such that $w_{a,b}\neq w_{a',b'}$ and $a$ belongs to the $i\th$ $\alpha$-region and $a'$ belongs to the ${i'}\th$ $\alpha$-region.  We need to characterize when
	\begin{equation}\label{eq:target}
		w_{a',b'} \leq {w_{a',b'}}_{*}\vee_{\alpha} w_{a,b}.
	\end{equation}
	For simplicity, let us write $w=w_{a,b}$, $w'=w_{a',b'}$ and $w'_{*}={w_{a',b'}}_{*}$.  Let $z=w'_{*}\vee_{\alpha}w$.  By definition, $\Inv(w'_{*})\cup\Inv(w)\subseteq\Inv(z)$.  By Corollary~\ref{cor:alpha_irreducibles}, $\Inv(w'_{*})=\Inv(w')\setminus\bigl\{(a',b')\bigr\}$.  Then, \eqref{eq:target} is satisfied if and only if $(a',b')\in\Inv(z)$, which by Lemma~\ref{lem:inversion_set_join} is the case if $(a',b')\in\Inv(w)$ or there exists $c\in\{a'{+}1,a'{+}2,\ldots,b'{-}1\}$ such that $(a',c)\in\Inv(w'_{*})$ and $(c,b')\in\Inv(w)$ or vice versa.
	
	Let us first consider the case where $a$ and $a'$ belong to the same $\alpha$-region, \ie $i=i'$.  There are two cases.
	
	(i) Let $a\leq a'$.  If $b'\leq b$, then Corollary~\ref{cor:alpha_irreducibles_comparison} implies $w'\leq_{L}w$ and \eqref{eq:target} holds. If $b<b'$, then by Corollary~\ref{cor:alpha_irreducibles}, $(a',b')\notin\Inv(w)$.  In fact, $(c,b')\notin\Inv(w)$ for any $c\in[n]$, and if $(a',c)\in\Inv(w)$, then $p_{i}+1\leq c\leq b$.  However, if $(c,b')\in\Inv(w'_{*})$, then $a'+1\leq c\leq p_{i}$.  Thus, $(a',b')\notin\Inv(z)$ so that \eqref{eq:target} is not satisfied.
	
	(ii) Let $a>a'$.  If $b\leq b'$, then Corollary~\ref{cor:alpha_irreducibles_comparison} implies $w<_L w'$, so that \eqref{eq:target} does not hold.  If $b>b'$, then by Corollary~\ref{cor:alpha_irreducibles}, $(a',b')\notin\Inv(w)$.  Again, $(a',c)\notin\Inv(w)$ for any $c\in[n]$, and if $(c,b')\in\Inv(w)$, then $a\leq c\leq p_{i}$.  However, if $(a',c)\in\Inv(w'_{*})$, then $p_{i}+1\leq c<b'$.  Thus, $(a',b')\notin\Inv(z)$ so that \eqref{eq:target} is not satisfied.
	
	Let us now consider the case where $a$ and $a'$ belong to different $\alpha$-regions, \ie $i\neq i'$.  By Corollary~\ref{cor:alpha_irreducibles}, $(a',b')\notin\Inv(w)$.  As before, we may actually conclude $(a',c)\notin\Inv(w)$ for all $c\in[n]$, and if $(c,b')\in\Inv(w)$, then $a\leq c\leq p_{i}$ and $p_{i}<b'\leq b$.  If $(a',c)\in\Inv(w'_{*})$, then $p_{i'}+1\leq c<b'$.  
	
	(i) If $i<i'$, then $p_{i}<p_{i'}+1$.  Thus, $(a',b')\notin\Inv(z)$ so that \eqref{eq:target} is not satisfied.
	
	(ii) If $i>i'$, then $a'<a$.  If $b'\leq p_{i}$, then $(c,b')\notin\Inv(w)$ for any $c\in[n]$ and \eqref{eq:target} cannot be satisfied.  If $p_{i}<b'$, then we may choose $c=a$ to see that $(a',b')\in\Inv(z)$ which implies \eqref{eq:target}.
\end{proof}

Figure~\ref{fig:tamari_galois} shows $\Galois\bigl(\Tamari\bigl((1,2,1)\bigr)\bigr)$ and $\Galois\bigl(\Tamari\bigl((2,1,2)\bigr)\bigr)$.  In \cite{thomas19rowmotion}*{Theorem~5.5} it was shown that the complement of the \emph{undirected} Galois graph of an extremal semidistributive lattice is precisely the $1$-skeleton of the canonical join complex.  By Proposition~\ref{prop:alpha_join_representations} the canonical join representations in $\Tamari(\alpha)$ correspond to noncrossing $\alpha$-partitions.  We thus have the following corollary (which may also be verified directly).

\begin{corollary}\label{cor:alpha_galois_edges_compatible}
	If there exists a directed edge $w_{a,b}\to w_{a',b'}$ in $\Galois\bigl(\Tamari(\alpha)\bigr)$, then the $\alpha$-arcs $(a,b)$ and $(a',b')$ are not compatible.
\end{corollary}

\begin{figure}
	\centering
	\begin{subfigure}[t]{.45\textwidth}
		\centering
		\includegraphics[scale=1,page=30]{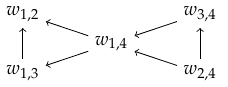}
		\caption{The Galois graph of $\Tamari\bigl((1,2,1)\bigr)$.}
		\label{fig:tamari_121_galois}
	\end{subfigure}
	\hspace*{1cm}
	\begin{subfigure}[t]{.45\textwidth}
		\centering
		\includegraphics[scale=1,page=25]{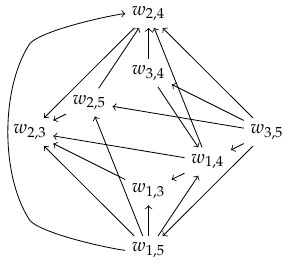}
		\caption{The Galois graph of $\Tamari\bigl((2,1,2)\bigr)$.}
		\label{fig:tamari_212_galois}
	\end{subfigure}
	\caption{Galois graphs of two parabolic Tamari lattices.}
	\label{fig:tamari_galois}
\end{figure}

\subsection{The topology of $\Tamari(\alpha)$}
	\label{sec:alpha_tamari_topology}
We conclude our study of $\Tamari(\alpha)$ with a topological characterization.  See Section~\ref{sec:poset_topology} for the necessary definitions.

\begin{theorem}\label{thm:alpha_tamari_topology}
	Let $n>0$ and let $\alpha$ be a composition of $n$.  Then, $\Tamari(\alpha)$ is spherical if and only if $\alpha=(n)$ or $\alpha=(1,1,\ldots,1)$.
\end{theorem}
\begin{proof}
	By Proposition~\ref{prop:alpha_tamari_trim}, $\Tamari(\alpha)$ is trim and therefore left-modular.  By Theorem~\ref{thm:left_modular_topology}, the order complex of the proper part of $\Tamari(\alpha)$ is a wedge of $k$ spheres, where $k=\bigl\lvert\mu\bigl(\Tamari(\alpha)\bigr)\bigr\rvert$.  By Corollary~\ref{cor:alpha_tamari_semidistributive}, $\Tamari(\alpha)$ is semidistributive, so that by Proposition~\ref{prop:meet_semidistributive_spherical}, $\mu\bigl(\Tamari(\alpha)\bigr)\in\{-1,0,1\}$.  
	
	If $\alpha=(n)$, then $\Symmetric_{\alpha}(231)=\{\id\}$, and $\Tamari(\alpha)$ is thus the singleton lattice, so that $\mu\bigl(\Tamari(\alpha)\bigr)=1$.  Otherwise, let $\alpha=(\alpha_{1},\alpha_{2},\ldots,\alpha_{r})$ with $r>1$.  Let $A$ denote the set of atoms of $\Tamari(\alpha)$, and let $B$ denote the canonical join representation of $\woa$.  By construction, 
	\begin{align*}
		A & = \Bigl\{w_{p_{1},p_{1}+1},w_{p_{2},p_{2}+1},\ldots,w_{p_{r-1},p_{r-1}+1}\Bigr\},\\
		B & = \Bigl\{w_{1,p_{2}},w_{p_{1}+1,p_{3}},\ldots,w_{p_{r-2}+1,n}\Bigr\}.
	\end{align*}
	Then, $\mu\bigl(\Tamari(\alpha)\bigr)=(-1)^{n}$ if and only if the join of all atoms of $\Tamari(\alpha)$ is $\woa$ if and only if $A=B$ if and only if $r=n$, $p_{1}=1$ and $p_{i+1}=p_{i}+1$ for $i\in[r{-}1]$ if and only if $\alpha=(1,1,\ldots,1)$.  
\end{proof}

We conclude the proof of Theorem~\ref{thm:alpha_tamari_core_label_order}.

\begin{proof}[Proof of Theorem~\ref{thm:alpha_tamari_core_label_order}]
	By Theorem~\ref{thm:alpha_tamari_intersection_property}, $\Tamari(\alpha)$ has the intersection property, which---by Theorem~\ref{thm:intersection_property}---implies that $\CLO\bigl(\Tamari(\alpha)\bigr)$ is a meet-semilattice.  Clearly, a meet-semilattice is a lattice if and only if it has a greatest element.  By \cite{muehle19the}*{Lemma~4.6}, the core label order of a congruence-uniform lattice $\Lattice$ has a greatest element if and only if $\Lattice$ is spherical.  By Theorem~\ref{thm:alpha_tamari_topology}, $\Tamari(\alpha)$ is spherical if and only if $\alpha=(n)$ or $\alpha=(1,1,\ldots,1)$.
\end{proof}

\section{Parabolic Chapoton triangles}
	\label{sec:alpha_chapoton}
We end our study of the $\alpha$-Tamari lattices with an enumerative observation.  Let us consider the \defn{$M_{\alpha}$-triangle}, \ie the (bivariate) generating function of the M{\"o}bius function of $\CLO\bigl(\Tamari(\alpha)\bigr)$ with respect to the number of descents:
\begin{equation}\label{eq:m_triangle}
	M_{\alpha}(x,y) \defs \sum_{u,v\in\Symmetric_{\alpha}(231)}\mu_{\CLO(\Tamari(\alpha))}(u,v)x^{\lvert\Des(u)\rvert}y^{\lvert\Des(v)\rvert}.
\end{equation}

\begin{example}\label{ex:triangles_1}
	The core label orders of $\Tamari\bigl((1,2,1)\bigr)$ and $\Tamari\bigl((2,1,2)\bigr)$ are shown in Figures~\ref{fig:tamari_121_clo} and \ref{fig:tamari_212_clo}, respectively.  We may compute the corresponding $M_{\alpha}$-triangles directly:
	\begin{align*}
		M_{(1,2,1)}(x,y) & =  4x^{2}y^{2} - 9xy^{2} + 5y^{2} + 5xy - 5y + 1,\\
		M_{(2,1,2)}(x,y) & = x^{3}y^{3} - 4x^{2}y^{3} + 5xy^{3} - 2y^{3} + 9x^{2}y^{2} - 22xy^{2} + 13y^{2} + 8xy - 8y + 1.
	\end{align*}
\end{example}

The motivation for the consideration of the $M_{\alpha}$-triangle comes from \cite{chapoton04enumerative}, where the corresponding polynomial for $\alpha=(1,1,\ldots,1)$ was introduced.  More precisely, F.~Chapoton considered the generating function of the M{\"o}bius function of the noncrossing partition lattice with respect to the number of bumps.  In view of Theorems~\ref{thm:nc_231_bijection} and \ref{thm:alpha_tamari_clo_nc_poset}, Chapoton's $M$-triangle agrees with our $M_{(1,1,\ldots,1)}$-triangle.

One of Chapoton's central observations in \cites{chapoton04enumerative,chapoton06sur} is the fact that $M_{(1,1,\ldots,1)}(x,y)$ behaves extremely well under certain (invertible) variable substitutions.  If $\alpha=(\alpha_{1},\alpha_{2},\ldots,\alpha_{r})$, then we define the \defn{$H_{\alpha}$}- and the \defn{$F_{\alpha}$-triangle} as follows:
\begin{align}
	H_{\alpha}(x,y) & \defs \bigl(x(y-1)+1\bigr)^{r-1}M_{\alpha}\left(\frac{y}{y-1},\frac{x(y-1)}{x(y-1)+1}\right),\label{eq:h_triangle}\\
	F_{\alpha}(x,y) & \defs y^{r-1}M_{\alpha}\left(\frac{y+1}{y-x},\frac{y-x}{y}\right)\label{eq:f_triangle}.
\end{align}

\begin{example}\label{ex:triangles_2}
	Continuing Example~\ref{ex:triangles_1}, we obtain
	\begin{align*}
		H_{(1,2,1)}(x,y) & = x^{2}y^{2} + 2x^{2}y + x^{2} + 2xy + 3x + 1,\\
		H_{(2,1,2)}(x,y) & = \frac{x^{3}y^{3} + x^{3}y^{2} + 3x^{3}y + 3x^{2}y^{2} - 4x^{3} + 6x^{2}y + 3xy + 5x + 1}{x(y-1)+1},
	\end{align*}
	and 
	\begin{align*}
		F_{(1,2,1)}(x,y) & = 5x^{2} + 4xy + y^{2} + 9x + 4y + 4,\\
		F_{(2,1,2)}(x,y) & = \frac{2x^{3} + 12x^{2}y + 4xy^{2} + y^{3} + 5x^{2} + 20xy + 4y^{2} + 4x + 8y + 1}{y}.
	\end{align*}
\end{example}

Computer experiments suggest the following conjecture.

\begin{conjecture}\label{conj:hf_triangles}
	Let $n>0$ and let $\alpha$ be a composition of $n$ into $r$ parts.  The rational functions $H_{\alpha}(x,y)$ and $F_{\alpha}(x,y)$ are polynomials with nonnegative integer coefficients if and only if $\alpha$ has at most one part exceeding $1$.
\end{conjecture}

If we replace the exponent $r-1$ in the definition of $H_{\alpha}(x,y)$ and $F_{\alpha}(x,y)$ by 
\begin{displaymath}
	d = \max_{w\in\Symmetric_{\alpha}(231)}\bigl\lvert\Des(w)\bigr\rvert,
\end{displaymath}
then we can verify directly that \eqref{eq:h_triangle} and \eqref{eq:f_triangle} produce polynomials with integer coefficients; however these coefficients need not all be nonnegative, as can be witnessed in the example $\alpha=(2,1,2)$.

If $\alpha$ has at most one part exceeding $1$, then Conjecture~\ref{conj:hf_triangles} implies the existence of combinatorial families $A_{H;\alpha}$ and $A_{F;\alpha}$, and combinatorial statistics $\sigma_{1}$, $\sigma_{2}$, $\tau_{1}$, $\tau_{2}$, such that
\begin{align*}
	H_{\alpha}(x,y) & = \sum_{a\in A_{H;\alpha}}x^{\sigma_{1}(a)}y^{\sigma_{2}(a)},\\
	F_{\alpha}(x,y) & = \sum_{a\in A_{F;\alpha}}x^{\tau_{1}(a)}y^{\tau_{2}(a)}.
\end{align*}

We close by suggesting candidates for $A_{H;\alpha}$ and $\sigma_{1},\sigma_{2}$.  For $n>0$, we define
\begin{align*}
	S_{n} \defs \bigl\{(i,i+1)\mid 1\leq i<n\bigr\},\\
	T_{n} \defs \bigl\{(i,j)\mid 1\leq i<j\leq n\bigr\},
\end{align*}
and we set $(i_{1},j_{1})\trianglelefteq (i_{2},j_{2})$ if and only if $i_{1}\geq i_{2}$ and $j_{1}\leq j_{2}$.

Let $\alpha=(\alpha_{1},\alpha_{2},\ldots,\alpha_{r})$ be a composition of $n$, and recall that $p_{i}=\alpha_{1}+\alpha_{2}+\cdots+\alpha_{i}$ for $i\in[r]$.  We now define
\begin{align*}
	S_{\alpha} & \defs \bigl\{(p_{i},p_{i}+1)\mid i\in[r-1]\bigr\},\\
	T_{\alpha} & \defs \bigl\{t\in T_{n}\mid s\trianglelefteq t\;\text{for some}\;s\in S_{\alpha}\}.
\end{align*}
In other words, $T_{\alpha}$ is the order filter generated by $S_{\alpha}$ in the poset $(T_{n},\trianglelefteq)$.  Let
\begin{displaymath}
	\NN(\alpha) \defs \bigl\{A\subseteq T_{\alpha}\mid A\;\text{is an antichain in}\;(T_{\alpha},\trianglelefteq)\bigr\}
\end{displaymath}
be the set of \defn{nonnesting $\alpha$-partitions}.  Figure~\ref{fig:root_poset_32122} shows a nonnesting $(3,2,1,2,2)$-partition.

\begin{figure}
	\centering
	\includegraphics[scale=.75,page=26]{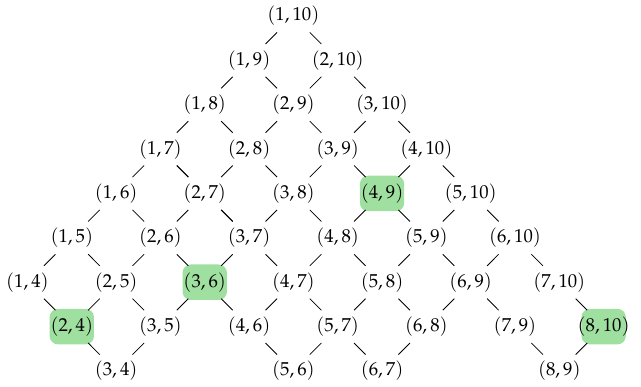}
	\caption{The poset $\bigl(T_{(3,2,1,2,2)},\trianglelefteq\bigr)$ with a nonnesting $(3,2,1,2,2)$-partition highlighted.}
	\label{fig:root_poset_32122}
\end{figure}

%
\begin{conjecture}\label{conj:h_triangle_combin}
	Let $n>0$, let $\alpha$ be a composition of $n$ and define 
	\begin{displaymath}
		\tilde{H}_{\alpha}(x,y) \defs \sum_{A\in\NN(\alpha)}x^{\lvert A\rvert}y^{\lvert A\cap S_{\alpha}\rvert}.
	\end{displaymath}
	Then, $H_{\alpha}(x,y)=\tilde{H}_{\alpha}(x,y)$ if and only if $\alpha$ has at most one part exceeding $1$.
\end{conjecture}

For $\alpha=(1,1,\ldots,1)$, Conjecture~\ref{conj:h_triangle_combin} follows from \cite{thiel14on}*{Theorem~2} and \cite{athanasiadis07on}*{Theorem~1.1}.  

\begin{example}\label{ex:triangles_3}
	We continue Examples~\ref{ex:triangles_1} and \ref{ex:triangles_2}.  Figures~\ref{fig:nonnesting_121} and \ref{fig:nonnesting_212} show the nonnesting $\alpha$-partitions for $\alpha=(1,2,1)$ and $\alpha=(2,1,2)$, respectively.  Whenever minimal elements are involved in an antichain, they are marked in red.  We obtain 
	\begin{align*}
		\tilde{H}_{(1,2,1)}(x,y) & = x^{2}y^{2} + 2x^{2}y + x^{2} + 2xy + 3x + 1,\\
		\tilde{H}_{(2,1,2)}(x,y) & = x^{2}y^{2} + x^{3} + 2x^{2}y + 6x^{2} + 2xy + 6x + 1,
	\end{align*}
	which supports Conjecture~\ref{conj:h_triangle_combin}.
\end{example}

\begin{figure}
	\centering
	\includegraphics[scale=1,page=27]{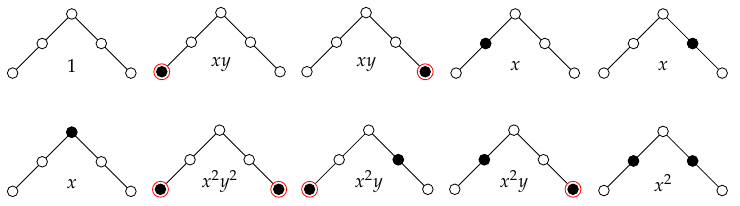}
	\caption{The nonnesting $(1,2,1)$-partitions.  Minimal elements are circled in red, and each time the term contributed to $\tilde{H}(x,y)_{(1,2,1)}$ is indicated.}
	\label{fig:nonnesting_121}
\end{figure}

\begin{figure}
	\centering
	\includegraphics[scale=1,page=28]{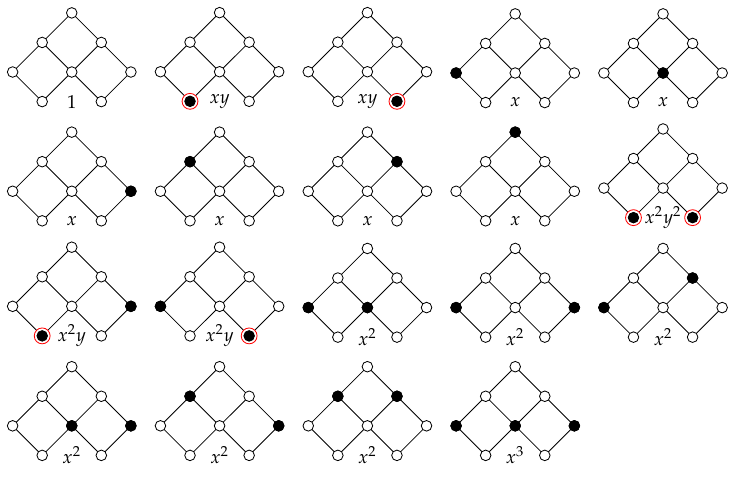}
	\caption{The nonnesting $(2,1,2)$-partitions.  Minimal elements are circled in red, and each time the term contributed to $\tilde{H}(x,y)_{(2,1,2)}$ is indicated.}
	\label{fig:nonnesting_212}
\end{figure}

\begin{remark}
	Apart from \cites{chapoton04enumerative,chapoton06sur}, analogues of the polynomials defined in \eqref{eq:m_triangle}--\eqref{eq:f_triangle} have appeared (in different contexts) in \cite{armstrong09generalized}*{Section~5.3} and \cite{garver17enumerative}*{Section~6}.
\end{remark}

\section*{Acknowledgements}

I thank Robin Sulzgruber for raising my interest in Chapoton triangles, and I am indebted to Christian Krattenthaler for his invaluable advice.

Moreover, I would like to express my gratitude towards an anonymous referee, who suggested that I narrow the focus of the paper to $\alpha$-Tamari lattices, provided the counterexample in Example~\ref{ex:alpha_refinement_no_semilattice} and provided simplifications to the proofs of Propositions~\ref{prop:core_label_set_is_noncrossing_partition} and \ref{prop:alpha_tamari_extremal}.

\appendix

\section{Posets and lattices}
	\label{sec:posets_lattices}
\subsection{Basic notions}
	\label{sec:order_basics}
Let $\Poset=(P,\leq)$ be a partially ordered set, or \defn{poset} for short.  The \defn{dual} poset of $\Poset$ is $\Poset^{\dual}\defs(P,\geq)$.  In the remainder, we only consider finite posets.

Given $a,b\in P$ with $a\leq b$, the set $[a,b]\defs\{c\in P\mid a\leq c\leq b\}$ is an \defn{interval} of $\Poset$.  Two elements $a,b\in P$ form a \defn{cover relation} if $a<b$ and $[a,b]$ consists of two elements.  In that case, we usually write $a\lessdot b$, and we say that $a$ \defn{is covered by} $b$ and that $b$ \defn{covers} $a$.  The set of cover relations of $\Poset$ is denoted by $\Covers(\Poset)$.

An element $a\in P$ is \defn{minimal} (resp. \defn{maximal}) in $\Poset$ if $b\leq a$ (resp. $a\leq b$) implies $b=a$ for all $b\in P$.  If $\Poset$ has a unique minimal element (usually denoted by $\least$) and a unique maximal element (usually denoted by $\grtst$), then $\Poset$ is \defn{bounded}.  In a bounded poset, any element covering $\least$ (resp. covered by $\grtst$) is an \defn{atom} (resp. \defn{coatom}).  

A \defn{chain} (resp. \defn{antichain}) of $\Poset$ is a subset of $P$ in which every two distinct elements are comparable (resp. incomparable).  A chain consisting of $k$ elements is also called a $k$-chain.  A chain is \defn{saturated} if it can be written as a sequence of cover relations, and it is \defn{maximal} if it is saturated and contains a minimal and a maximal element.

When $M$ is a set, then a map $f\colon\Covers(\Poset)\to M$ is an \defn{edge-labeling} of $\Poset$.  If $C:a_{0}\lessdot a_{1}\lessdot\cdots\lessdot a_{s}$ is a saturated chain, then we write
\begin{displaymath}
	f(C) \defs \bigl(f(a_{0},a_{1}),f(a_{1},a_{2}),\ldots,f(a_{s-1},a_{s})\bigr).
\end{displaymath}


An \defn{order ideal} (resp. \defn{order filter}) of $\Poset$ is a subset $I\subseteq P$ such that for all $a\in I$ if $b\leq a$ (resp. $a\leq b$) then $b\in I$.  An order ideal (resp. order filter) is \defn{principal} if it has a unique maximal (resp. minimal) element.

If for all $a,b\in P$ there exists a least upper bound $a\vee b$ (resp. a greatest lower bound $a\wedge b$), then $\Poset$ is a \defn{join-semilattice} (resp. \defn{meet-semilattice}).  If it exists, then $a\vee b$ (resp. $a\wedge b$) is the \defn{join} (resp. the \defn{meet}) of $a$ and $b$.  A poset that is both a join- and a meet-semilattice is a \defn{lattice}.  Every finite lattice is a bounded poset.

\subsection{Congruence-uniform lattices}
	\label{sec:congruence_uniform_lattices}
Let $\Lattice=(L,\leq)$ be a finite lattice.  A \defn{lattice congruence} is an equivalence relation $\Theta$ on $L$ such that for all $a,b,c,d\in L$ if $[a]_{\Theta}=[c]_{\Theta}$ and $[b]_{\Theta}=[d]_{\Theta}$, then $[a\vee b]_{\Theta}=[c\vee d]_{\Theta}$ and $[a\wedge b]_{\Theta}=[c\wedge d]_{\Theta}$.  The set $\Con(\Lattice)$ of all lattice congruences on $\Lattice$ forms a distributive lattice under refinement~\cite{funayama42on}.  If $a\lessdot b$ in $\Lattice$, then we denote by $\cg(a,b)$ the finest lattice congruence on $\Lattice$ in which $a$ and $b$ are equivalent.

An element $j\in L\setminus\{\least\}$ is \defn{join irreducible} if whenever $j=a\vee b$, then $j\in\{a,b\}$.  \defn{Meet-irreducible} elements can be defined dually.  Let $\JI(\Lattice)$ (resp. $\MI(\Lattice)$) denote the set of join-irreducible (resp. meet-irreducible) elements of $\Lattice$.  If $\Lattice$ is finite and $j\in\JI(\Lattice)$ (resp. $m\in\MI(\Lattice)$), then there exists  a unique element $j_{*}\in L$ (resp. $m^{*}\in L$) such that $j_{*}\lessdot j$ (resp. $m\lessdot m^{*}$).  If $j\in\JI(\Lattice)$, then $\cg(j)\defs\cg(j_{*},j)$.

\begin{theorem}[\cite{freese95free}*{Theorem~2.30}]\label{thm:irreducible_congruences}
	Let $\Lattice$ be a finite lattice and let $\Theta\in\Con(\Lattice)$.  The following are equivalent.
	\begin{enumerate}[\rm (i)]
		\item $\Theta$ is join-irreducible in $\Con(\Lattice)$.
		\item $\Theta=\cg(a,b)$ for some $(a,b)\in\Covers(\Lattice)$.
		\item $\Theta=\cg(j)$ for some $j\in\JI(\Lattice)$.
	\end{enumerate}
\end{theorem}

A consequence of Theorem~\ref{thm:irreducible_congruences} is the existence of a surjective map
\begin{displaymath}
	\cg_{*}\colon\JI(\Lattice)\to\JI\bigl(\textsl{\textsf{Con}}(\Lattice)\bigr), \quad j\mapsto\cg(j).
\end{displaymath}
A finite lattice is \defn{congruence uniform} if $\cg_{*}$ is a bijection for both $\Lattice$ and $\Lattice^{\dual}$.  In that case, Theorem~\ref{thm:irreducible_congruences} implies the existence of an edge-labeling
\begin{equation}\label{eq:cu_labeling}
	\lambda\colon\Covers(\Lattice)\to\JI(\Lattice), \quad (a,b)\mapsto j,
\end{equation}
where $j$ is the unique join-irreducible element of $\Lattice$ with $\cg(j)=\cg(a,b)$.

The cover relations in a congruence-uniform lattice with the same label under $\lambda$ can be characterized as follows.  Two cover relations $(a,b),(c,d)\in\Covers(\Lattice)$ are \defn{perspective} if either $a\vee d=b$ and $a\wedge d=c$ or $b\vee c=d$ and $b\wedge c=a$.  In that case, we write $(a,b)\perspective (c,d)$.  See Figure~\ref{fig:perspective_covers} for an illustration.

\begin{figure}
	\centering
	\includegraphics[scale=1,page=13]{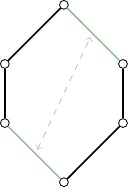}
	\caption{The green edges indicate a pair of perspective cover relations in the hexagon lattice.}
	\label{fig:perspective_covers}
\end{figure}

\begin{lemma}[\cite{garver18oriented}*{Lemma~2.6}]\label{lem:perspective_covers_same_label}
	Let $\Lattice$ be a congruence uniform lattice with labeling $\lambda$.  For $(a,b)\in\Covers(\Lattice)$ and $j\in\JI(\Lattice)$, $\lambda(a,b)=j$ if and only if $(a,b)\perspective (j_{*},j)$.
\end{lemma}

This lemma has the following consequences.

\begin{corollary}\label{cor:minimal_cover_irreducible}
	Let $\Lattice$ be congruence-uniform and $(a,b)\in\Covers(\Lattice)$ and $j\in\JI(\Lattice)$.  If $\lambda(a,b)=j$, then $j\leq b$.
\end{corollary}
\begin{proof}
	By Lemma~\ref{lem:perspective_covers_same_label}, $(a,b)\perspective (j_{*},j)$, which implies that either $b\leq j$ or $j\leq b$.  If $b<j$, then $b\vee j_{*}=j$, which implies that $b\not\leq j_{*}$.  This means that there is a saturated chain from $b$ to $j$ which does not contain $j_{*}$; in particular $j$ has at least two lower covers.  This contradicts $j\in\JI(\Lattice)$.
\end{proof}

\begin{corollary}\label{cor:cu_labeling_no_duplicates}
	For any saturated chain $C$ of a congruence-uniform lattice $\Lattice$, the sequence $\lambda(C)$ does not contain duplicate entries.
\end{corollary}
\begin{proof}
	Let $C:a_{0}\lessdot a_{1}\lessdot\cdots\lessdot a_{s}$ be a saturated chain of $\Lattice$, and pick some index $i\in[s]$ such that $\lambda(a_{i-1},a_{i})=k\in\JI(\Lattice)$.  By Corollary~\ref{cor:minimal_cover_irreducible}, $k\leq a_{i}$.  Thus, for any $j\geq i$ it follows that $k\leq a_{j}$, and thus $k\vee a_{j}=a_{j}\lessdot a_{j+1}$.  Hence, $(k_{*},k)$ and $(a_{j},a_{j+1})$ are not perspective, and thus $\lambda(a_{j},a_{j+1})\neq k$ by Lemma~\ref{lem:perspective_covers_same_label}.
\end{proof}

\subsection{The core label order of a congruence-uniform lattice}
	\label{sec:core_label_order}
The labeling \eqref{eq:cu_labeling} of a congruence-uniform lattice $\Lattice=(L,\leq)$ gives rise to an alternate way of ordering $L$.  This order was first considered by N.~Reading in connection with posets of regions of simplicial hyperplane arrangements under the name \emph{shard intersection order}; see \cite{reading16lattice}*{Section 9-7.4}.  

For $a\in L$, we define its \defn{nucleus} by
\begin{displaymath}
	a_{\downarrow} \defs \bigwedge_{b\in L\colon b\lessdot a}b,
\end{displaymath}
and we define the \defn{core label set} of $a$ by
\begin{equation}\label{eq:core_label_set}
	\Psi_{\Lattice}(a) \defs \Bigl\{\lambda(b,c)\mid a_{\downarrow}\leq b\lessdot c\leq a\Bigr\}.
\end{equation}
If no confusion can arise, we drop the subscript ``$\Lattice$'' from the core label set.  The \defn{core label order} of $\Lattice$ is the poset $\CLO(\Lattice)\defs\bigl(L,\sqsubseteq\bigr)$, where $a\sqsubseteq b$ if and only if $\Psi(a)\subseteq\Psi(b)$.  See Figure~\ref{fig:core_label_order} for an illustration.

\begin{figure}
	\centering
	\begin{subfigure}[t]{.45\textwidth}
		\centering
		\includegraphics[scale=1,page=14]{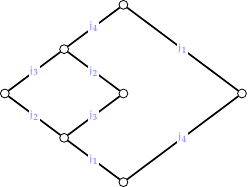}
		\caption{A congruence-uniform lattice with edge labeling $\lambda$.}
		\label{fig:cu_lattice}
	\end{subfigure}
	\hspace*{1cm}
	\begin{subfigure}[t]{.45\textwidth}
		\centering
		\includegraphics[scale=1,page=15]{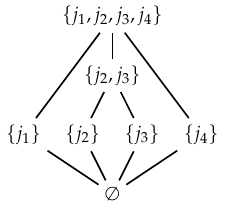}
		\caption{The core label order of the lattice from Figure~\ref{fig:cu_lattice}.  We have represented the elements by their core label sets.}
		\label{fig:cu_lattice_clo}
	\end{subfigure}
	\caption{The core label order of a congruence-uniform lattice.}
	\label{fig:core_label_order}
\end{figure}

Moreover, $\Lattice$ has the \defn{intersection property} if for all $a,b\in L$ there exists $c\in L$ such that $\Psi(a)\cap\Psi(b)=\Psi(c)$.  

\begin{theorem}[\cite{muehle19the}*{Theorem~4.8}]\label{thm:intersection_property}
	A finite, congruence-uniform lattice $\Lattice$ has the intersection property if and only if $\CLO(\Lattice)$ is a meet-semilattice.
\end{theorem}

If $\Theta$ is a lattice congruence on $\Lattice$, then a join-irreducible element $j\in\JI(\Lattice)$ is \defn{contracted} by $\Theta$ if $[j_{*}]_{\Theta}=[j]_{\Theta}$.  

\begin{lemma}[\cite{muehle19the}*{Lemma~4.9}]\label{lem:core_label_sets_quotients}
	Let $\Lattice=(L,\leq)$ be a finite, congruence-uniform lattice and let $\Theta\in\Con(\Lattice)$.  Let $\Sigma$ be the set of join-irreducible elements of $\Lattice$ contracted by $\Theta$.  For $a\in L$, the core label set $\Psi_{\Lattice/\Theta}\bigl([a]_{\Theta}\bigr)$ is in bijection with $\Psi_{\Lattice}(a)\setminus\Sigma$.
\end{lemma}

Moreover, a lattice congruence $\Theta$ on $\Lattice$ induces a canonical projection map
\begin{displaymath}
	\proj_{\Theta}\colon L\to L, \quad a\mapsto\min[a]_{\Theta},
\end{displaymath}
which identifies the quotient lattice $\Lattice/\Theta$ of $\Lattice$ by $\Theta$ with the restriction of $\Lattice$ to the minimal elements in the congruence classes of $\Theta$.  Then, Lemma~\ref{lem:core_label_sets_quotients} can be rephrased as: 
\begin{displaymath}
	\Psi_{\Lattice/\Theta}\bigl(\proj_{\Theta}(a)\bigr) = \Psi_{\Lattice}(a)\setminus\Sigma.
\end{displaymath}

\begin{proposition}[\cite{muehle19the}*{Proposition~4.11}]\label{prop:intersection_property_quotients}
	The intersection property is inherited by quotient lattices.
\end{proposition}

\subsection{Semidistributive lattices}
	\label{sec:semidistributive_lattices}
A finite lattice $\Lattice=(L,\leq)$ is \defn{join semidistributive} if for all $a,b,c\in L$ with $a\vee b=a\vee c$ follows $a\vee(b\wedge c)=a\vee b$.  We may define \defn{meet-semidistributive} lattices dually.  A lattice is \defn{semidistributive} if it is both join and meet semidistributive.

\begin{theorem}[\cite{day79characterizations}*{Theorem~4.2}]\label{thm:congruence_uniform_semidistributive}
	Every congruence-uniform lattice is semidistributive.
\end{theorem}

The converse of Theorem~\ref{thm:congruence_uniform_semidistributive} is not true, see for instance \cite{nation00unbounded}*{Section~3}.  Join-semidistributive lattices have another characteristic property: every element can be represented canonically as the join of a particular set of join-irreducible elements.

More precisely, a subset $A\subseteq L$ is a \defn{join representation} of $a\in L$ if $a=\bigvee A$.  A join representation is \defn{irredundant} if there is no proper subset of $A$ that joins to $a$.  For two irredundant join representations $A_{1}$ and $A_{2}$ of $a$ we say that $A_{1}$ refines $A_{2}$ if for every $a_{1}\in A_{1}$ there exists some $a_{2}\in A_{2}$ with $a_{1}\leq a_{2}$.  In other words, the order ideal generated by $A_{1}$ is contained in the order ideal generated by $A_{2}$.  A join representation of $a$ is \defn{canonical} if it is irredundant and refines every other irredundant join representation of $a$.  If a canonical join representation of $a$ exists, then it is an antichain of join-irreducible elements; the \defn{canonical joinands} of $a$.

\begin{theorem}[\cite{freese95free}*{Theorem~2.24}]\label{thm:join_semidistributive_canonical_representation}
	A finite lattice is join semidistributive if and only if every element admits a canonical join representation.
\end{theorem}

Figure~\ref{fig:extremal_not_semidistributive} shows a lattice that is not join semidistributive, because the top element does not have a canonical join representation.  Indeed, there are two irredundant join representations of the top element: the three atoms and the two highlighted elements, but none of these sets refines the other.

\begin{figure}
	\centering
	\includegraphics[scale=1,page=9]{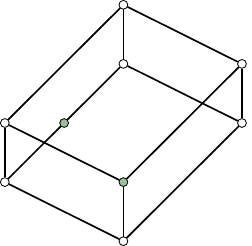}
	\caption{A lattice that is not join semidistributive.}
	\label{fig:extremal_not_semidistributive}
\end{figure}

Proposition~2.2 in \cite{reading15noncrossing} states that in any finite lattice $\Lattice$, every subset of a canonical join representation is again a canonical join representation.  Thus, the set of canonical join representations of $\Lattice$ forms a simplicial complex; the \defn{canonical join complex} of $\Lattice$.  If $\Lattice$ is join semidistributive, then the faces of this complex are indexed by the elements of $\Lattice$.  See \cite{barnard19canonical} for more information on the canonical join complex.

If $\Lattice$ is congruence uniform, then we can use the labeling from \eqref{eq:cu_labeling} to compute canonical join representations in $\Lattice$.

\begin{theorem}[\cite{garver18oriented}*{Proposition~2.9}]\label{thm:join_representation_lower_covers}
	Let $\Lattice=(L,\leq)$ be a finite, congruence-uniform lattice.  The canonical join representation of $a\in L$ is $\bigl\{\lambda(b,a)\mid b\lessdot a\bigr\}$.
\end{theorem}

\begin{corollary}\label{cor:canonical_joinands_lower_covers}
	Let $\Lattice=(L,\leq)$ be a finite, congruence-uniform lattice.  The number of canonical joinands of $a\in L$ equals the number of elements covered by $a$ in $\Lattice$.
\end{corollary}
\begin{proof}
	Let $a\in L$.  Let $b_{1},b_{2}\in L$ be such that $b_{1}\neq b_{2}$ and $\lambda(b_{1},a)=k=\lambda(b_{2},a)$.  By Corollary~\ref{cor:minimal_cover_irreducible}, $k\leq a$, and by Lemma~\ref{lem:perspective_covers_same_label}, $k\wedge b_{1}=k_{*}=k\wedge b_{2}$.  Since $\Lattice$ is semidistributive, it follows that $k_{*}=k\wedge(b_{1}\vee b_{2})=k\wedge a=k$, a contradiction.
	
	The claim now follows from Theorem~\ref{thm:join_representation_lower_covers}.
\end{proof}

\begin{proposition}[\cite{barnard18coxeter}*{Proposition~4.11}]\label{prop:canonical_joinands_quotient}
	Let $\Lattice$ be a finite, join-semidistributive lattice and let $\Theta$ be a lattice congruence on $\Lattice$.  If $j$ is a canonical joinand of $a\in L$ such that $j$ is not contracted by $\Theta$, then $j$ is a canonical joinand of $\proj_{\Theta}(a)$ in $L$.  Moreover, if $\proj_{\Theta}(a)=a$, then none of the canonical joinands of $a$ is contracted by $\Theta$.
\end{proposition}

Let $j\in\JI(\Lattice)$.  If the set $\{a\in L\mid j_{*}\leq a, j\not\leq a\}$ has a greatest element, then we denote it by $\kappa(j)$.  Whenever $\kappa(j)$ exists, it must be meet irreducible.  We recall two facts about the partial map $\kappa\colon\JI(\Lattice)\to\MI(\Lattice)$.

\begin{lemma}[\cite{freese95free}*{Corollary~2.55}]\label{lem:kappa_bijection}
	If $\Lattice$ is finite and semidistributive, then $\kappa$ is a bijection.  Thus, $\bigl\lvert\JI(\Lattice)\bigr\rvert=\bigl\lvert\MI(\Lattice)\bigr\rvert$.
\end{lemma}

\begin{lemma}[\cite{freese95free}*{Lemma~2.57}]\label{lem:kappa_comparison}
	Let $\Lattice=(L,\leq)$ be a finite lattice, and let $j\in\JI(\Lattice)$ be such that $\kappa(j)$ exists.  For every $a\in L$ we have $a\leq\kappa(j)$ if and only if $j\not\leq j_{*}\vee a$.
\end{lemma}

\subsection{Trim Lattices}
	\label{sec:trim_lattices}
Let $\Lattice=(L,\leq)$ be a finite lattice.  The \defn{length} of a chain of $\Lattice$ is one less than its cardinality.  Let $\ell(\Lattice)$ denote the maximum length of a maximal chain of $\Lattice$.  For every finite lattice, the following holds:
\begin{equation}\label{eq:lattice_irreducibles_length}
	\ell(\Lattice)\leq\min\Bigl\{\bigl\lvert\JI(\Lattice)\bigr\rvert,\bigl\lvert\MI(\Lattice)\bigr\rvert\Bigr\}.
\end{equation}
If these three quantities are the same, \ie if $\bigl\lvert\JI(\Lattice)\bigr\rvert=\ell(\Lattice)=\bigl\lvert\MI(\Lattice)\bigr\rvert$, then $\Lattice$ is \defn{extremal}~\cite{markowsky92primes}.  It follows from \cite{markowsky92primes}*{Theorem~14(ii)} that any finite lattice can be embedded as an interval into an extremal lattice.  Consequently, extremality is not inherited by intervals.

In \cite{thomas06analogue}, a strengthening of extremality was introduced which does have this hereditary property.  An element $a\in L$ is \defn{left modular} if for all $b,c\in L$ with $b<c$ it holds that
\begin{displaymath}
	(b\vee a)\wedge c = b\vee(a\wedge c).
\end{displaymath}
If $\Lattice$ has a maximal chain of length $\ell(\Lattice)$ consisting entirely of left-modular elements, then $\Lattice$ is \defn{left modular}.  An extremal, left-modular lattice is \defn{trim}~\cite{thomas06analogue}.  

It was recently shown that any extremal, semidistributive lattice is already trim.

\begin{theorem}[\cite{thomas19rowmotion}*{Theorem~1.4}]\label{thm:semidistributive_extremal_trim}
	Every extremal semidistributive lattice is trim.
\end{theorem}

Figure~\ref{fig:extremal_not_semidistributive} shows the smallest extremal lattice that is not left modular.  It has only one chain of maximum length, but the non-atom marked in green is not left modular.

\subsection{The Galois graph of an extremal lattice}
	\label{sec:galois_graph}
Extremal lattices can be compactly represented in terms of a directed graph---the \defn{Galois graph}---which encodes the incomparability relation between join- and meet-irreducible elements~\cite{markowsky92primes}*{Theorem~11}.  

Let $\Lattice=(L,\leq)$ be extremal with $\ell(L)=n$, and fix a maximal chain $C:\least=a_{0}\lessdot a_{1}\lessdot\cdots\lessdot a_{n}=\grtst$.  Then, $\bigl\lvert\JI(\Lattice)\bigr\rvert=n=\bigl\lvert\MI(\Lattice)\bigr\rvert$.  We can label the join-irreducible elements by $j_{1},j_{2},\ldots,j_{n}$ and the meet-irreducible elements by $m_{1},m_{2},\ldots,m_{n}$ such that 
\begin{equation}\label{eq:extremal_order}
	j_{1}\vee j_{2}\vee\cdots\vee j_{s} = a_{s} = m_{s+1}\wedge m_{s+2}\wedge\cdots\wedge m_{n}
\end{equation}
for all $s$.  We can always order some of the irreducibles in such a way; the extremality guarantees that this is an ordering of \emph{all} irreducibles.

Using this order, we define the \defn{Galois graph} of $\Lattice$ following \cite{markowsky92primes}*{Theorem~2(b)} (see also \cite{thomas19rowmotion}*{Section~2.3}).  This is the directed graph $\Galois(\Lattice)$ with vertex set $[n]$, where $s\to t$ if and only if $s\neq t$ and $j_{s}\not\leq m_{t}$.  Figure~\ref{fig:extremal_not_congruence_uniform_galois_graph} shows the Galois graph of the extremal lattice in Figure~\ref{fig:extremal_not_congruence_uniform}.

\begin{figure}
	\centering
	\begin{subfigure}[t]{.3\textwidth}
		\centering
		\includegraphics[scale=1,page=10]{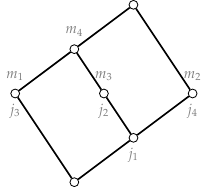}
		\caption{The smallest extremal lattice that is not congruence uniform.  The join- and meet-irreducible elements are labeled according to \eqref{eq:extremal_order}.}
		\label{fig:extremal_not_congruence_uniform}
	\end{subfigure}
	\hspace*{.5cm}
	\begin{subfigure}[t]{.25\textwidth}
		\centering
		\includegraphics[scale=1,page=11]{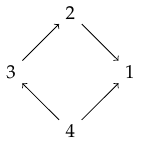}
		\caption{The Galois graph of the lattice in Figure~\ref{fig:extremal_not_congruence_uniform}.}
		\label{fig:extremal_not_congruence_uniform_galois_graph}
	\end{subfigure}
	\hspace*{.5cm}
	\begin{subfigure}[t]{.25\textwidth}
		\centering
		\includegraphics[scale=1,page=12]{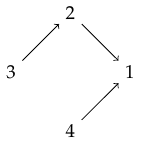}
		\caption{Another directed graph defined by the join-irreducible elements of the lattice in Figure~\ref{fig:extremal_not_congruence_uniform}.}
		\label{fig:extremal_not_congruence_uniform_other_graph}
	\end{subfigure}
	\caption{The Galois graph of an extremal lattice that is not congruence uniform.}
	\label{fig:galois_extremal_not_congruence_uniform}
\end{figure}

The Galois graph of an extremal lattice $\Lattice$ uniquely determines $\Lattice$ as we will briefly outline next.  In general, let $G=\bigl([n],E\bigr)$ be a directed simple graph.  An \defn{orthogonal pair} of $G$ is a pair $(A,B)$ with $A,B\subseteq [n]$, $A\cap B=\emptyset$ and there is no $(s,t)\in E$ with $s\in A$, $t\in B$.  An orthogonal pair is \defn{maximal} if both sets $A$ and $B$ are (cardinality-wise) maximal with that property.  We may define a partial order on the set of maximal orthogonal pairs of $G$ by setting $(A_{1},B_{1})\sqsubseteq (A_{2},B_{2})$ if and only if $A_{1}\subseteq A_{2}$ (or equivalently $B_{2}\subseteq B_{1}$).  The set of maximal orthogonal pairs of $G$ with respect to this partial order is a lattice.  Extremal lattices may now be characterized via this construction; see also \cite{thomas19rowmotion}*{Section~2.3}.

\begin{theorem}[\cite{markowsky92primes}*{Theorem~11}]\label{thm:extremal_lattice_representation}
	Every finite extremal lattice is isomorphic to the lattice of maximal orthogonal pairs of its Galois graph.  Conversely, given any directed graph $G=\bigl([n],E\bigr)$ such that $(s,t)\in E$ only if $s>t$, the lattice of maximal orthogonal pairs is extremal.
\end{theorem}

\begin{figure}
	\centering
	\begin{subfigure}[t]{.45\textwidth}
		\centering
		\includegraphics[scale=1,page=31]{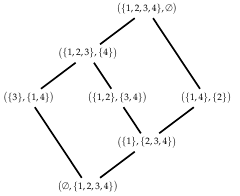}
		\caption{The lattice of maximal orthogonal pairs of the graph in Figure~\ref{fig:extremal_not_congruence_uniform_galois_graph}.}
		\label{fig:lattice_orthogonal_pairs_1}
	\end{subfigure}
	\hspace*{1cm}
	\begin{subfigure}[t]{.45\textwidth}
		\centering
		\includegraphics[scale=1,page=32]{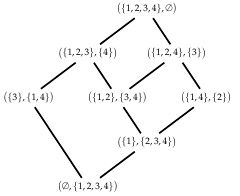}
		\caption{The lattice of maximal orthogonal pairs of the graph in Figure~\ref{fig:extremal_not_congruence_uniform_other_graph}.}
		\label{fig:lattice_orthogonal_pairs_2}
	\end{subfigure}
	\caption{Two lattices of maximal orthogonal pairs.}
	\label{fig:lattice_orthogonal_pairs}
\end{figure}

Figure~\ref{fig:lattice_orthogonal_pairs} shows two lattices of maximal orthogonal pairs.  Constructing the Galois graph of an extremal lattice requires to understand the incomparability relation between join- and meet-irreducible elements.  If an extremal lattice is additionally congruence uniform, we may simplify this construction by only taking the join-irreducible elements into account.

If $\Lattice$ is both extremal and congruence uniform, then we may define another ordering of the join-irreducible elements using the labeling $\lambda$ from \eqref{eq:cu_labeling}.  In particular, we pick a maximal chain of maximum length, and we order the join-irreducible elements according to the order in which they appear in $\lambda(C)$.  This is a total order of all join-irreducible elements of $\Lattice$ by Corollary~\ref{cor:cu_labeling_no_duplicates} and the assumption that $\Lattice$ is extremal.

\begin{lemma}\label{lem:cu_order_is_extremal_order}
	Let $\Lattice$ be a finite, extremal and congruence-uniform lattice, and fix a maximal chain $C$ of maximum length.  The ordering of $\JI(\Lattice)$ coming from \eqref{eq:extremal_order} agrees with the order in which the join-irreducible elements appear in $\lambda(C)$.  
\end{lemma}
\begin{proof}
	Let $\ell(\Lattice)=n$, and pick a maximal chain $C:\least=a_{0}\lessdot a_{1}\lessdot\cdots\lessdot a_{n}=\grtst$.  Suppose that---with respect to $C$---the order of the join-irreducible elements of $\Lattice$ from \eqref{eq:extremal_order} is $j_{1},j_{2},\ldots,j_{n}$.
	
	It follows that $a_{1}=j_{1}$, and therefore $\lambda(a_{0},a_{1})=j_{1}$.  Now let $t\in[n]$ and suppose that $\lambda(a_{s-1},a_{s})=j_{s}$ for all $s\leq t$.  Let $\lambda(a_{t},a_{t+1})=j$.  By Corollary~\ref{cor:cu_labeling_no_duplicates}, $j\in\JI(\Lattice)\setminus\{j_{1},j_{2},\ldots,j_{t}\}$.  By \eqref{eq:extremal_order}, $a_{t+1}=a_{t}\vee j_{t+1}$, and by Lemma~\ref{lem:perspective_covers_same_label} and Corollary~\ref{cor:minimal_cover_irreducible}, $a_{t+1}=a_{t}\vee j$.
	
	Since \eqref{eq:extremal_order} determines a linear order on $\JI(\Lattice)$, it follows that $j=j_{t}$.  The claim follows by induction.
\end{proof}

\begin{corollary}\label{cor:kappa_extremal_order}
	Let $\Lattice$ be a finite, extremal and congruence-uniform lattice, in which $\JI(\Lattice)$ and $\MI(\Lattice)$ are ordered as in \eqref{eq:extremal_order} with respect to some maximal chain of length $n=\ell(\Lattice)$.  
	\begin{enumerate}[\rm (i)]
		\item For $s\in[n]$, $m_{s}=\kappa(j_{s})$.
		\item For $s,t\in[n]$, $j_{s}\not\leq m_{t}$ if and only if $s\neq t$ and $j_{t}\leq {j_{t}}_{*}\vee j_{s}$.
		\item If $j_{t}\leq j_{s}$, then there is a directed edge from $s$ to $t$ in $\Galois(\Lattice)$.
	\end{enumerate}
\end{corollary}
\begin{proof}
	(i) Let $\ell(\Lattice)=n$, and let $s\in[n]$.  By Lemma~\ref{lem:kappa_bijection}, there exists $m=\kappa(j_{s})$.  By definition of $\kappa$, ${j_{s}}_{*}\leq m$ and $j_{s}\not\leq m$ and $m$ is maximal with this property.  Thus $j_{s}\wedge m={j_{s}}_{*}$ and $j_{s}\vee m=m^{*}$.  Thus $({j_{s}}_{*},j_{s})\perspective (m,m^{*})$, and Lemma~\ref{lem:perspective_covers_same_label} implies $\lambda({j_{s}}_{*},j_{s})=\lambda(m,m^{*})$.
	
	Let $C:\least=a_{0}\lessdot a_{1}\lessdot\cdots\lessdot a_{n}=\grtst$.  By Lemma~\ref{lem:cu_order_is_extremal_order}, $\lambda(a_{s-1},a_{s})=\lambda({j_{s}}_{*},j_{s})=\lambda(m,m^{*})$, and applying Lemma~\ref{lem:perspective_covers_same_label} once more yields $(a_{s-1},a_{s})\perspective(m,m^{*})$.  Since $m$ is meet irreducible, we conclude---by the dual statement of Corollary~\ref{cor:minimal_cover_irreducible}---that $a_{s-1}=m\wedge a_{s}$.  By \eqref{eq:extremal_order}, $m=m_{s}$.
	
	(ii) This follows from Lemma~\ref{lem:kappa_comparison} and (i).
	
	(iii) This follows from (ii) and the definition of $\Galois(\Lattice)$.
\end{proof}

If $\Lattice$ is extremal and congruence uniform, then Corollary~\ref{cor:kappa_extremal_order} implies that we may view $\Galois(\Lattice)$ as a directed graph with vertex set $\JI(\Lattice)$ where we have a directed edge $j_{s}\to j_{t}$ if and only if $s\neq t$ and $j_{t}\leq {j_{t}}_{*}\vee j_{s}$.  For extremal lattices that are not congruence uniform, this construction normally yields a directed graph that is not isomorphic to $\Galois(\Lattice)$.  This is illustrated in Figure~\ref{fig:galois_extremal_not_congruence_uniform}.  Figure~\ref{fig:extremal_not_congruence_uniform_other_graph} shows the directed graph with vertex set $[4]$ and a directed edge $s\to t$ if and only if $s\neq t$ and $j_{t}\leq {j_{t}}_{*}\vee j_{s}$ holds in the extremal lattice in Figure~\ref{fig:extremal_not_congruence_uniform}.

\subsection{Poset topology}
	\label{sec:poset_topology}
The \defn{order complex} $\Delta(\Poset)$ of a finite poset $\Poset=(P,\leq)$ is the simplicial complex whose faces are the chains of $\Poset$.  If $\Poset$ has a least or a greatest element, then $\Delta(\Poset)$ is always contractible.  If $\Poset$ is bounded, then we denote by $\overline{\Poset}\defs\bigl(P\setminus\{\least,\grtst\},\leq)$ the \defn{proper part} of $\Poset$. 

The \defn{M{\"o}bius function} of $\Poset$ is the map $\mu_{\Poset}\colon P\times P\to\mathbb{Z}$, inductively defined by $\mu_{\Poset}(a,a)\defs 1$ for all $a\in P$ and by
\begin{displaymath}
	\mu_{\Poset}(a,b) \defs -\sum_{c\in P\colon a\leq c<b}\mu_{\Poset}(a,c),
\end{displaymath}
for all $a,b\in P$ with $a\neq b$.  If $\Poset$ is bounded, then the \defn{M{\"o}bius number} of $\Poset$ is $\mu(\Poset)\defs\mu_{\Poset}(\least,\grtst)$. 

It follows from a result of P.~Hall that the M{\"o}bius number of $\Poset$ equals the reduced Euler characteristic of $\Delta(\overline{P})$; see \cite{stanley11enumerative_vol1}*{Proposition~3.8.5}.

Let us recall two results concerning the M{\"o}bius number of certain kinds of lattices.  The first one follows from \cite{liu99left} (see also \cite{mcnamara06poset}*{Theorem~8}) and \cite{bjorner96shellable}*{Theorem~5.9}.  

\begin{theorem}\label{thm:left_modular_topology}
	Let $\Lattice$ be a finite, left-modular lattice.  The order complex $\Delta(\overline{\Lattice})$ is homotopic to a wedge of $\bigl\lvert\mu(\Lattice)\bigr\rvert$-many spheres.
\end{theorem}

Consequently, if $\Lattice$ is left modular with $\mu(\Lattice)\in\{-1,1\}$, then $\Delta(\overline{\Lattice})$ is a sphere, and we call $\Lattice$ \defn{spherical}.

\begin{proposition}[\cite{muehle19the}*{Proposition~2.13}]\label{prop:meet_semidistributive_spherical}
	Let $\Lattice$ be a finite, meet-semidistributive lattice with $n$ atoms.  If the join of all atoms of $\Lattice$ is $\grtst$, then $\mu(\Lattice)=(-1)^{n}$.  Otherwise, $\mu(\Lattice)=0$.
\end{proposition}


\end{document}